\documentclass[a4paper,12pt,oneside]{amsart}
\usepackage{amsmath,amsfonts,amssymb,amsthm,amscd,latexsym,euscript,xypic}
\usepackage[all]{xy}

\usepackage{pdfsync}

\newtheorem{theorem}{Theorem}[section]
\newtheorem{claim}[theorem]{Claim}
\newtheorem{corollary}[theorem]{Corollary}
\newtheorem{definition}[theorem]{Definition}

\newtheorem{proposition}[theorem]{Proposition}
\newtheorem{prop}[theorem]{Proposition}

\newcommand{\Ae}{{\mathcal A}}

\newcommand{\Be}{{\mathcal B}}

\newcommand{\Ce}{{\mathcal C}}

\newcommand{\Ee}{{\mathcal E}}

\newcommand{\Pe}{{\mathbb P}}
\newcommand{\Rcal}{{\mathcal R}}

\newcommand{\Ue}{{\mathcal U}}

\newcommand{\OR}{{\rm Ord}}
\newcommand{\rank}{{\rm rank}}

\renewcommand{\emptyset}{\varnothing}
\newcommand{\card}{C\!ARD}

\numberwithin{equation}{section}

\newcommand{\SR}{{\rm SR}}
\newcommand{\GSR}{{\rm GSR}}

\newcommand{\SGSR}{{\rm SGSR}}

\newcommand{\Rg}{\mathop{R\hspace{-0.8pt}g}}
\newcommand{\WI}{\mathop{W\!I}}

\newcommand{\WC}{\mathop{W\!C}}
\newcommand{\PW}{\mathop{P\hspace{-0.8pt}w\hspace{-0.8pt}S\hspace{-0.8pt}e\hspace{-0.8pt}t}}
\newcommand{\Cd}{\mathop{C\hspace{-0.8pt}d}}
\newcommand{\VP}{{\rm VP}}

\newcommand{\Coll}{{\rm Coll}}

\newcommand{\PRP}{\rm{PRP}}
\newcommand{\PSR}{\rm{PSR}}
\newcommand{\SPSR}{\rm{SPSR}}
\newcommand{\ESR}{\rm{ESR}}

\newcommand{\crit}{{\rm crit}}

\newcommand{\range}{{\rm range}}

\newcommand{\lt}[1]{{\smalllt}#1}
\newcommand{\smalllt}{\mathrel{\mathchoice{\raise2pt\hbox{$\scriptstyle<$}}{\raise1pt\hbox{$\scriptstyle<$}}{\raise0pt\hbox{$\scriptscriptstyle<$}}{\scriptscriptstyle<}}}

\newcommand{\seq}[2]{\langle{#1}~\vert~{#2}\rangle}

\newcommand{\calL}{\mathcal{L}}

\begin{document}

\title{Large Cardinals as Principles of Structural reflection}
\author{Joan Bagaria}
\address{ICREA (Instituci\'o Catalana de Recerca i Estudis Avan\c{c}ats) and
\newline \indent Departament de Matem\`atiques i Inform\`atica, Universitat de Barcelona. 
Gran Via de les Corts Catalanes, 585,
08007 Barcelona, Catalonia.}
\email{joan.bagaria@icrea.cat}
\thanks{Part of this research was  supported by  the Spanish
Government under grant MTM2017-86777-P,
and by the Generalitat de Catalunya (Catalan Government) under grant SGR 270-2017. }
\date{\today }

\begin{abstract}
After discussing  the limitations inherent to all set-theoretic reflection principles akin to those studied by  A. L\'evy et. al. in the 1960's, we  introduce new principles of reflection based on the general notion of \emph{Structural Reflection} and argue that they are in strong agreement with the conception of reflection implicit in  Cantor's original idea of the unknowability of the \emph{Absolute}, which was subsequently developed  in the works of Ackermann,  L\'evy, G\"odel, Reinhardt, and others. We then present a comprehensive survey of results  showing that different forms of the new principles of Structural Reflection are equivalent to well-known large cardinals axioms covering all regions of the large-cardinal hierarchy, thereby justifying the naturalness of the latter.
\end{abstract}

\maketitle




In the framework of Zermelo-Fraenkel (ZF) set theory\footnote{For all undefined set-theoretic notions see \cite{Jech} or \cite{Kan:THI}.} the universe $V$ of all sets is usually represented as a stratified cumulative hierarchy of sets  indexed by the ordinal numbers. Namely, $V=\bigcup_{\alpha \in OR}V_{\alpha}$, where

\begin{enumerate}
\item[] $V_0=\emptyset$
\item[] $V_{\alpha +1}=\mathcal{P}(V_\alpha),\;  \mbox{ namely the set of all subsets of }V_{\alpha}, \mbox{ and} $
\item[] $V_{\lambda}=\bigcup_{\alpha <\lambda}V_{\alpha},\mbox{ if $\lambda$ is a limit ordinal.}$
\end{enumerate}
This view of the set-theoretic universe justifies \emph{a posteriori} the ZF axioms. For not only all ZF axioms  are  true in $V$, but they are also necessary to build $V$. Indeed, the axioms of Extensionality, pairing, Union, Power-set, and Separation are used to define the set-theoretic operation given by $$G(x)=\bigcup \{ \mathcal{P}(y): \exists z (\langle z,y\rangle \in x)\}.$$
The axiom of Replacement is then needed to prove by transfinite recursion that the operation $V$ on the ordinals given by $V(\alpha) = G(V\restriction \alpha)$ is well-defined and unique. Then $V(\alpha)=V_\alpha$ is as defined above. The two remaining ZF axioms are easily justified:  the axiom of Infinity leads the ordinal  sequence into the transfinite, and  is the essence of set theory, for its negation (namely, ZFC minus Infinity plus the negation of Infinity) yields a  theory mutually interpretable with Peano Arithmetic; and the axiom of Regularity simply says, in the presence of the other ZF axioms, that there are no more sets beyond those in $V$. 

In the context of ZFC (ZF plus the Axiom of Choice), the \emph{Axiom of Choice} is clearly true in $V$, although it is not necessary to build $V$.

In ZF, or ZFC, other representations of $V$ as a  cumulative hierarchy are  possible. By  ``a cumulative hierarchy"  we mean a   union of sets $X_{\alpha}$, defined by transfinite recursion on   a \emph{club}, i.e., closed and unbounded, class $C$ of ordinals $\alpha$, such that:
\begin{enumerate}
\item[] $\alpha \leq \beta \mbox{ implies }X_{\alpha}\subseteq X_{\beta},$
\item[] $X_{\lambda}=\bigcup_{\alpha <\lambda}X_{\alpha},\mbox{ if $\lambda$ is a limit point of $C$, and}$
\item[] $V=\bigcup_{\alpha\in C}X_{\alpha}.$
\end{enumerate}
For example, let $H_\kappa$, for $\kappa$ an infinite cardinal, be the set of all sets whose transitive closure has cardinality less than $\kappa$. 
Then the $H_\kappa$ also form a cumulative hierarchy indexed by the club class $\card$ of infinite cardinal numbers.
%
(Note however  that to prove $\bigcup_{\kappa \in \card}H_\kappa =V$ one needs the Axiom of Choice to guarantee  that every (transitive) set has a cardinality, and therefore every set belongs to some $H_\kappa$.)
Nevertheless, all representations of $V$ as a cumulative hierarchy  are essentially the same, in the following sense:
Suppose $V=\bigcup_{\alpha\in C}X_{\alpha}$ and $V=\bigcup_{\alpha \in D}Y_{\alpha}$ are two cumulative hierarchies, where $C$ and $D$ are club classes of ordinals. Then there is a club class of ordinals $E\subseteq C\cap D$ such that $X_{\alpha}=Y_{\alpha}$, for all $\alpha \in E$.

Every representation of $V$ as a cumulative hierarchy is subject to the \emph{reflection} phenomenon. Namely, the fact that every sentence of the first-order language of set theory that holds in $V$, holds already at some stage of the hierarchy.
Indeed, the \emph{Principle of Reflection}  of Montague and L\'evy (\cite{Le:AS}), provable in $ZF$, asserts that every formula of the first-order language of set theory true in $V$ holds in some   $V_{\alpha}$. In fact, for every formula $\varphi (x_1,\ldots ,x_n)$ of the language of set theory, ZF proves that  there exists an ordinal $\alpha$ such that for every $a_1,\ldots ,a_n\in V_\alpha$, 
$$\varphi (a_1,\ldots ,a_n)\quad \mbox{   if and only if   }\quad V_{\alpha}\models \varphi(a_1,\ldots ,a_n).$$
Even  more is true: for each natural number $n$, there is a $\Pi_n$ definable club  proper class $C^{(n)}$ of ordinals such that ZF proves that for every $\kappa \in C^{(n)}$, every $\Sigma_n$ formula $\varphi(x_1,\ldots ,x_n)$, and every  $a_1,\ldots ,a_n\in V_{\kappa}$,
$$\varphi (a_1,\ldots ,a_n)\quad \mbox{   if and only if   }\quad V_{\kappa}\models \varphi(a_1,\ldots ,a_n).$$
That is, $(V_{\kappa},\in)$ is a $\Sigma_n$-elementary substructure of $(V,\in)$, henceforth written as  $(V_{\kappa},\in)\preceq_{n}(V,\in)$, or simply as $V_\kappa \preceq_n V$. 

\smallskip

The import of the Principle of Reflection is highlighted by the result of L\'evy \cite{Le:AS} showing that the ZF axioms of Extensionality, Separation, and Regularity, together with the principle of Complete Reflection (CR) imply ZF. The CR principle is the schema asserting that for every formula $\varphi (x_1,\ldots ,x_n)$ of the language of set theory there exists a transitive set $A$ closed under subsets (i.e, all subsets of elements of $A$ belong to $A$) such that for every $a_1,\ldots ,a_n\in A$, 
$$\varphi (a_1,\ldots ,a_n)\leftrightarrow A\models \varphi(a_1,\ldots ,a_n).$$ Since the $V_\alpha$ are transitive and closed under subsets, this shows that the reflection phenomenon, as expressed by the CR principle or  the Reflection Theorem (given the existence of the $V_\alpha$'s), is not only deeply ingrained in the ZF axioms, but it captures the main content of ZF.

In the series of papers \cite{Le:AS, Le:PR,  Le:OPR}   L\'evy considers stronger principles of reflection, formulated as axiom schemata, and he shows them equivalent to the existence of inaccessible, Mahlo, and Hyper-Mahlo cardinals. The unifying idea behind such principles  is clearly stated by L\'evy at the beginning of \cite{Le:PR}: 
\begin{quote}
If we start with the idea of the impossibility of distinguishing, by specific means, the universe from partial universes we shall be led to the following axiom schemata, listed according to increasing strength. These axiom schemata will be called \emph{principles of reflection} since they state the existence of standard models (by models we shall mean, for the time being, models whose universes are sets) which reflect in some sense the state of the universe.
\end{quote}
 This idea of reflection, namely the impossibility of distinguishing the universe from its partial universes (such as the $V_\alpha$),  is also implicit in earlier work of Ackermann (\cite{Ac:ZAM}), but it is L\'evy who demonstrates how it can be used to find new natural theories strengthening ZF. Indeed,  
 in his review of Levy's article \cite{Le:OPR}, Feferman \cite{Fe:RevLe} writes:
\begin{quote}
The author's earlier work demonstrated very well that the diversity of known set-theories could be viewed with more uniformity in the light of various reflection principles, and that these also provided a natural way to ``manufacture" new theories. The present paper can only be regarded as a beginning of a systematic attempt to compare the results.
\end{quote}


More recently, it has been argued by many authors that any intrinsic justification of new set-theoretic axioms, beyond ZF, and in particular the axioms of large cardinals, should be based on stronger forms of L\'evy's reflection principles. In the next section we shall see some of these arguments, as well as the  limitations, clearly exposed by Koellner \cite{Ko:RP}, inherent to all reflection principles akin to those studied by  L\'evy. In Section 2 we will introduce new principles of reflection based on what we call \emph{Structural Reflection} and will argue that they are in strong agreement with the notion of reflection implicit in the original idea of Cantor's of the unknowability of the \emph{absolute}, which was subsequently developed  in the works of Ackermann,  L\'evy, G\"odel, Reinhardt, and others. The rest of the paper, starting with Section 3, will present a series of results  showing that different forms of the new principles of Structural Reflection are equivalent to well-known large cardinals axioms covering all regions of the large-cardinal hierarchy, thereby justifying the naturalness of the latter.
\footnote{The work presented in the following sections started over ten years ago. After a talk I gave in Barcelona in 2011 on large cardinals as principles of structural reflection, John Baldwin, who attended the talk and was at the time editor of the BSL, encouraged me to write a survey article for the Bulletin. Well, here it is. I'm thankful to John for the invitation and I apologise to him for the long delay.}

\section{Set Theoretic axioms as reflection principles}

Cantor (\cite[p.~205, note 2]{C:Grund}) emphasizes the unknowability of the transfinite sequence of all ordinal numbers, which he thinks of as an ``appropriate symbol of the absolute":

\begin{quote}
The absolute can only be acknowledged, but never known, not even 
approximately known. 
\end{quote}
This principle  of the unknowability of the absolute, which in Cantor's work seems to have only a metaphysical  (non-mathematical) meaning (see \cite{Jane}), resurfaces again in the 1950's in the work of Ackermann and L\'evy, taking the  mathematical form of a principle of reflection. Thus, in Ackermann's set theory -- in fact, a theory of classes -- which is formulated in the first-order language of set theory with an additional constant symbol for the class $V$ of all sets, the idea of reflection is expressed in the form of an axiom schema of comprehension:

\begin{quote}
\emph{Ackermann's Reflection:} Let $\varphi(x, z_1, \ldots, z_n)$  be a formula which does not contain the constant symbol $V$. Then for every $\vec{a}=a_1,\ldots ,a_n \in V$,
    $$\forall x (\varphi (x, \vec{a}) \rightarrow x \in V) \rightarrow \exists y (y \in V \land \forall x (x \in y \leftrightarrow \varphi (x, \vec{a}))).$$ 
    \end{quote}
    A consequence of Ackermann's Reflection is that no formula can define $V$, or the class OR of all ordinal numbers, and is therefore in agreement with Cantor's principle of the unknowability of the absolute. 
However, Ackermann's set theory (with Foundation) was shown by L\'evy \cite{Le:AST} and Reinhardt \cite{re:AST} to be essentially equivalent to ZF, in the sense that both theories are equiconsistent and prove  the same theorems about sets. Thus Ackermann's set theory did not provide any real advantage with respect to the simpler and intuitively clearer ZF axioms, and so it was eventually forgotten.
 

Later on, in the context of the wealth of independence results in set theory that were obtained starting in the mid-1960's thanks to the forcing technique, and as a result of the subsequent need for the identification of new set-theoretic axioms, G\"odel (as quoted by  Wang \cite{Wa:LJ}), places Ackermann's  principle (stated in a Cantorian, non-mathematical form)  as the main source for new set-theoretic axioms beyond ZFC:

\begin{quote}
All the principles for setting up the axioms of set theory should be reducible to a form of Ackermann's principle: The Absolute is unknowable. The strength of this principle increases as we get stronger and stronger systems of set theory. The other principles are only heuristic principles. Hence, the central principle is the reflection principle, which presumably will be understood better as our experience increases.
\end{quote}
Thus, according to G\"odel,  the fundamental guiding principle in setting up new axioms of set theory is the unknowability of the absolute, and so any new axiom should be based on such principle. G\"odel's program consisted, therefore,  in formulating stronger and stronger systems of set theory by adding to the base theory, which presumably could be taken as ZFC, new principles akin to  Ackermann's.  

So the question is how should one understand and formulate  the idea of reflection embodied in Ackermann's principle. 
Some light is provided by G\"odel in the following quote from \cite[p.~285]{Wa:LJ},
where he asserts that the indefinability of $V$ should be the source of all axioms of infinity, i.e., all large-cardinal axioms.  

\begin{quote}
Generally I believe that, in the last analysis, every axiom of infinity should be derivable from the (extremely plausible) principle that $V$ is indefinable, where definability is to be taken in [a] more and more generalized and idealized sense. 
\end{quote}
One possible interpretation of G\"odel's principle of the indefinability of $V$ is as an unrestricted version of the Montague-L\'evy Principle of Reflection. Namely: every formula, with parameters, in \emph{any} formal language with the membership relation, that holds in $V$, must also hold in some $V_{\alpha}$. 
This has been indeed the usual way to interprete G\"odel's view of reflection as a justification for the axioms of large cardinals. 
This is made explicit in  Koellner \cite[p.~208]{Ko:RP}, where he identifies reflection principles with generalized forms of the Montague-L\'evy Principle of Reflection. Namely,
\begin{quote}
Reflection principles aim to articulate the informal idea that the height of the universe is ``absolute infinite" and hence cannot be ``characterized from below". These principles assert that any statement true in $V$ is true in some smaller $V_\alpha$. 
\end{quote}
Moreover, he  explicitly interpretes G\"odel's view of reflection as a source of large cardinals in this way (\cite[p.~208]{Ko:RP}):
\begin{quote}
Since the most natural way to assert that $V$ is indefinable is via reflection principles and since to assert this in a ``more and more generalized and idealized sense" is to move to languages of higher-order with higher-order parameters, G\"odel is (arguably) espousing the view that higher-order reflection principles imply all large cardinal axioms. 
\end{quote}
This kind of reflection, namely generalised forms of the  Principle of Reflection  of L\'evy and Montague  allowing for the reflection of second-order formulas, has been used to justify   weak large-cardinal axioms, such as the existence of inaccessible, Mahlo, or even weakly-compact cardinals. Let us see briefly some examples to illustrate how the arguments work.

\medskip

The  Principle of Reflection also holds with proper classes as additional predicates. Namely, for every (definable, with set parameters) proper class $A$, and every $n$, there is an ordinal $\alpha$ such that 
$$(V_{\alpha},\in, A\cap V_{\alpha})\preceq_{n}(V,\in ,A).$$
Thus, if we  interpret second-order quantifiers over $V$ as ranging only over definable classes, then for each $n$,  the following second-order sentence, call it $\varphi_n$, is true in $V$
$$\forall A\exists \alpha (V_{\alpha},\in, A\cap V_{\alpha})\preceq_{n}(V,\in ,A).$$
Now, applying reflection, $\varphi_n$ must reflect to some $V_{\kappa}$. The second-order universal quantifier in $\varphi_n$ is now interpreted in $V_\kappa$, hence ranging over \emph{all} subsets of $V_\kappa$,  which are now available in $V_{\kappa +1}$. So we have that for each subset $A$ of $V_{\kappa}$, there is some $\alpha <\kappa$ such that 
$$(V_{\alpha},\in, A\cap V_{\alpha})\preceq_{n}(V_{\kappa},\in ,A).$$
If this is so, and even just for $n=1$, then $\kappa$ must be an inaccessible cardinal, and this property actually characterises  inaccessible cardinals \cite{Le:AS}. 
Thus the existence of an inaccessible cardinal follows rather easily from the reflection of the single $\Pi^1_1$ sentence $\varphi_1$ to some $V_\kappa$. 

\medskip

Further, the following  stronger form of reflection is provable in ZFC (as a schema): if $C$ is a definable club proper class of ordinals, then for every proper class $A$, and every $n$, there is a club proper class of $\alpha$ in $C$ such that 
$$(V_{\alpha},\in, A\cap V_{\alpha})\preceq_{n}(V,\in ,A).$$
Hence, for each $n$,  the following $\Pi^1_1$ sentence, with $C$ as a second-order parameter,  holds in $V$ (again, interpreting the universal second-order quantifier as ranging over definable classes):
$$\forall A\forall \beta \exists \alpha >\beta (\alpha \in C \wedge (V_{\alpha},\in, A\cap V_{\alpha})\preceq_{n}(V,\in ,A)).$$
Applying reflection, there is $\kappa$ such that for every subset $A$ of $V_\kappa$ there are unboundedly many $\alpha$ in $C\cap \kappa$ such that
$$(V_{\alpha},\in, A\cap V_{\alpha})\preceq_{n}(V_\kappa ,\in ,A).$$
But since $C$ is a club, $\kappa\in C$, and so (in the case of $n\geq 1$) $\kappa$ is an inaccessible cardinal in $C$.
Now consider the following $\Pi^1_1$ sentence, which we have just shown (using $\Pi^1_1$ reflection) to hold in $V$:
$$\forall C(C\mbox{ is a club subclass of ordinals}\rightarrow  \exists \kappa (\kappa \mbox{ innaccessible}\wedge \kappa \in C)).$$
Any cardinal that reflects the $\Pi^1_1$ sentence consisting of the conjunction of $\varphi_1$ above with the last displayed sentence is an inaccessible cardinal $\lambda$ with the property that every club subset of $\lambda$ contains an inaccessible cardinal, i.e., $\lambda$ is a Mahlo cardinal.

\medskip

Furthermore, if we are willing to assume that all $\Pi^1_1$ sentences with second-order parameters reflect, we may as well reflect this property of $V$. So, consider the following $\Pi^1_2$ sentence, which says that $V$ reflects all $\Pi^1_1$ sentences with parameters,
$$\forall A\forall \varphi \in \Pi^1_1 \exists \alpha ((V,\in ,A)\models \varphi \rightarrow (V_\alpha ,\in ,A\cap V_\alpha )\models \varphi ).$$
If $V_\kappa$ reflects this sentence, then $\kappa$ is a   $\Pi^1_1$-indescribable cardinal, i.e.,  a  weakly-compact cardinal (see \cite{Jech}).
Similar arguments, applied to sentences of order $n$, but only allowing first-order and  second-order parameters, yield the existence of $\Sigma^m_n$ and $\Pi^m_n$ indescribable cardinals. 

\medskip

While the arguments just given may seem reasonable, or even natural, we do not think they provide a justification for the existence of the large cardinals obtained in that way. The problem is that the relevant second-order statements  are true in $V$ only when interpreting second-order variables as ranging over \emph{definable} classes, yet when the statements are reflected to some $V_\alpha$ the second-order variables are reinterpreted as ranging over the full power-set of $V_\alpha$. It is precisely this transition from definable classes to the full power-set that yields the large-cardinal strength. Thus, the  fundamental objection is that  second-order  reflection from $V$ to some $V_\alpha$, or to some set, is always problematic because so is unrestricted second-order quantification over $V$, as the full \emph{power-class} of $V$ is not available.

\medskip

Nevertheless, even if one is willing to accept the existence of cardinals $\kappa$ that reflect  $n$-th--order sentences, with parameters of order greater than $2$, one does not obtain  large cardinals much stronger than the indescribable ones. Indeed, to start with, and as first noted in \cite{Re:Rem}, one cannot even have reflection for $\Pi^1_1$ sentences with unrestricted third-order parameters.  For suppose, towards a contradiction, that  $\kappa$ is a cardinal that reflects such sentences. Let $A$ be the collection $\{ V_{\alpha}:\alpha <\kappa\}$  taken as a third-order parameter, i.e., as a subset of $\mathcal{P}(V_{\kappa})$. Then the $\Pi^1_1$ sentence  
$$\forall X\exists x(X\in A\to X=x)$$
where $X$ is a second-order variable and $x$ is first-order, 
 asserts that every element of $A$ is a set. The sentence is clearly true in $(V_{\kappa},\in ,A)$, but false in any $(V_{\alpha},\in ,A\cap \mathcal{P}(V_{\alpha}))$ with $\alpha <\kappa$, because $V_\alpha$ belongs to $A\cap \mathcal{P}(V_{\alpha})$ but is not an element of $V_\alpha$.

One possible way around the problem of second-order reflection with third-order parameters is to allow such parameters, or even higher-order parameters, but to restrict the kind of sentences to be reflected. This is the approach taken by Tait \cite{Ta:Goedel}. He considers the class $\Gamma^{(2)}$ of formulas which, in normal form, have all universal quantifiers restricted to first-order and second-order variables and the only atomic formulas allowed to appear negated are either those of first order or of the form $x\in X$, where $x$ is a variable of first order and $X$ a variable of second order. Tait shows that reflection at some $V_\kappa$ for the class of $\Gamma^{(2)}$ sentences, allowing parameters of arbitrarily-high finite order,  implies that $\kappa$ is an ineffable cardinal (see \cite{Ko:RP}), and that $V_\kappa$ reflects all such sentences whenever $\kappa$ is a measurable cardinal.
A sharper upper bound on the consistency strength of this kind of reflection is given by Koellner  \cite[Theorem 9]{Ko:RP}. He shows that below the first $\omega$-Erd\"os cardinal, denoted by $\kappa (\omega)$,  there exists a cardinal $\kappa$ such that $V_\kappa$ reflects all $\Gamma^{(2)}$ formulas. The existence of $\kappa (\omega)$ is, however,  a rather mild large-cardinal assumption, since it is compatible with $V=L$.

At this point, the question is thus whether reflection can consistently hold (modulo large cardinals) for a wider class of sentences. But Koellner \cite{Ko:RP}, building on   some results of Tait, shows that no $V_\kappa$ can reflect the class of formulas of the form $\forall X\exists Y \varphi (X,Y, Z)$, where $X$ is of third-order, $Y$ is of any finite order, $Z$ is of fourth order, and $\varphi$ has only first-order quantifiers and  its only negated atomic subformulas  are either  of first order or of the form $x\in X$, where $x$ is of first order and $X$ is of second order.
Other kinds of restrictions on the class of sentences to be reflected are possible (see \cite{Mc:CP} for the consistency of some forms  of reflection slightly stronger than Tait's $\Gamma^{(2)}$), but Koellner \cite{Ko:RP} convincingly shows that the existence of a cardinal $\kappa$ such that $V_\kappa$ reflects any reasonable expansion of the class of sentences $\Gamma^{(2)}$, with parameters of order greater than 2,   either follows from the existence of $\kappa(\omega)$ or is outright inconsistent.

These results seem to put an end to the program of providing an \emph{intrinsic}\footnote{See \cite{Ko:RP} for a discussion on \emph{intrinsic} versus \emph{extrinsic} justification of the axioms of set theory.} justification of large-cardinal axioms, even for axioms as strong as the existence of $\kappa(\omega)$, by showing that their existence  follows from strong higher-order reflection properties holding at some $V_\kappa$. In particular, the program cannot even provide justification for the existence of measurable cardinals.
Thus, the conclusion is that if one understands reflection principles as asserting that some sentences (even of higher order, and with parameters) that hold in $V$ must hold in some $V_\alpha$, then reflection principles cannot be used to  justify the existence of large cardinals up to or beyond $\kappa (\omega)$. Moreover, as we already emphasized, a more fundamental problem with  the use of higher-order reflection principles is that either second-order quantification over $V$ is interpreted as ranging over definable classes, in which case second-order reflection does not yield any large cardinals unless one makes the dubious jump from definable classes to the full power-set, or is ill-defined, as the full power-class of $V$ does not exist.

\subsection{A remark on the undefinability of $V$}

Before we go on to propose a new kind of reflection principles, let us take a pause to  consider another possible interpretation of G\"odel's principle of the undefinability of $V$ as a justification for large-cardinal axioms.

The statement that a set $A$ is definable is usually understood in two different senses:
\begin{enumerate}
\item There is a formula $\varphi (x)$ that defines $A$. That is, for every set $a$, $a$ belongs to $A$ if and only if $\varphi (a)$ holds.
The formula $\varphi$ may have parameters, provided they are \emph{simpler} than $A$, e.g., their rank  is less than the rank of $A$.

\item $A$ is the unique solution of a formula $\varphi (x)$. Again, $\varphi$ may have parameters simpler than $A$.
\end{enumerate}
There is, however, no essential difference between (1) and (2). For the formula $\varphi (x)$ defines a set $A$ in the sense of (1) if and only if the formula $\forall x (x\in y \leftrightarrow \varphi (x))$ defines $A$ in the sense of (2).
But if $A$ is a proper class, then (1) and (2) are very different, even if only because (2) needs to be reformulated to make any sense.
 If we understand definability as in (1), then   there are many formulas that define $V$, for instance the formula $x=x$. So,  the notion of  indefinability of $V$ can only be  understood in the sense of (2), once properly reformulated. To express that  $V$ is not the unique solution of a formula, possibly with some sets as parameters, we need to make sense of the fact that a formula is true \emph{of} $V$, as opposed to being   true \emph{in} $V$. 

Let $\mathcal{L}_V$ be the first-order language of set theory expanded with a constant symbol $\bar{a}$ for every set $a$, and a new constant symbol $v$. Define the class $\mathcal{T}$ of sentences of $\mathcal{L}_V$ recursively as follows: $\varphi$ belongs to  $\mathcal{T}$ if and only if 
\begin{enumerate}
\item[] $\varphi$ is of the form $\bar{a}\in \bar{b}$ or $\bar{a}\in v$, for some sets $a, b$ such that $a\in b$, or
\item[] $\varphi$ is of the form $\bar{a}=\bar{a}$ for some set $a$, or $v=v$, or 
\item[] $\varphi \equiv \neg \psi$, and $\psi$ does not belong to $\mathcal{T}$, or
\item[] $\varphi\equiv \psi \wedge \theta$, and both $\psi$ and $\theta$ belong to $\mathcal{T}$, or
\item[] $\varphi \equiv \exists y \psi (y)$,  and there is a set $a$ such that $\psi (\bar{a})$ belongs to $\mathcal{T}$.
\end{enumerate}
The idea is that if $\varphi$ belongs to $\mathcal{T}$, then $\varphi$ is true in the structure $$\bar{V}:= \langle V\cup \{ V\}, \in , V, \langle \bar{a}\rangle_{a\in V}\rangle$$ where the constant $v$ is interpreted as $V$. And conversely, if $\varphi$ is a sentence in the language  $\mathcal{L}_V$ that is true in  $\bar{V}$, then $\varphi \in \mathcal{T}$.
Thus one may construe the principle of indefinability of $V$, as expressed in G\"odel's quote above, as follows

\begin{quotation}
\emph{Indefinability of $V$:} Every  $\varphi \in \mathcal{T}$  is true in some $$\bar{V}_\alpha : = \langle V_\alpha \cup \{ V_\alpha \}, \in , V_\alpha , \langle \bar{a}\rangle_{a\in V_\alpha}\rangle.$$
\end{quotation}
Of course, to express this principle in the first-order language of set theory one needs to do it as a schema. Namely, for each  $n$ let $\mathcal{T}_n$ be the class of $\Sigma_n$ sentences of $\mathcal{T}$, and  let 

\begin{quotation}
\emph{$\Sigma_n$-Indefinability of $V$:} Every   $\varphi \in \mathcal{T}_n$  is true in some $\bar{V}_\alpha$.
\end{quotation}

Given a $\Sigma_n$ sentence $\varphi$ of the language of set theory, where $n\geq 1$, if it is true in $V$, i.e., if $\models_n \varphi$ holds,  then the sentence $\varphi^v$ obtained from $\varphi$ by bounding all quantifiers by $v$ belongs to $\mathcal{T}_n$. Hence, by $\Sigma_n$\emph{-Indefinability of} $V$, $\varphi^v$ is true in some $\bar{V}_\alpha$, and therefore $V_\alpha \models \varphi$. Thus \emph{Indefinability of $V$} directly  implies the Principle of Reflection of Montague-L\'evy (over the theory ZF minus Infinity). Conversely, by  induction on the complexity of $\varphi$, it is easily shown that ZF proves $\Sigma_n$-\emph{Indefinability of $V$}. 
Thus, to derive stronger reflection principles based on the  indefinability of $V$ one needs to understand ``definability", following G\"odel's quote above, in a ``more and more generalized and idealized sense". One could expand the class $\mathcal{T}$  by adding higher-order sentences that are true of $V$. However this will not lead us very far. For if  $\mathcal{T}'$ is any reasonable class of (higher-order) sentences that are true of $V$, then \emph{Indefinability of $V$} for the class $\mathcal{T}'$ will imply  the reflection, in the sense of Montague-L\'evy,  of all sentences in $\mathcal{T}'$. Therefore, the limitations seen above of the extensions of the Montague-L\'evy Principle of Reflection to higher-order formulas  apply also to these generalized forms of \emph{Indefinability of $V$}.

\section{Structural Reflection}
\label{SR}
The main obstacle for  the program of finding an intrinsic justification of large-cardinal axioms via strong principles of reflection lies, we believe,  on a too restrictive interpretation of the notion of reflection, namely the interpretation investigated by Tait,  Koellner, and others, according to which the reflection properties of $V$ are exhausted by generalized forms of the Montague-L\'evy Principle of Reflection  to higher-order logics. 

Let us think again about the notion of reflection as derived from the Cantor-Ackermann principle of the unknowability of the absolute. A  different interpretation of this principle may be extracted from another claim made by G\"odel, as quoted in \cite{Wa:LJ}:  
\begin{quote}
The universe of sets cannot be uniquely characterized (i.e., distinguished from all its initial segments) by any internal structural property of the membership relation in it which is expressible in any logic of finite or transfinite type, including infinitary logics of any cardinal number.
\end{quote}
This quote does not immediately suggest that the uncharacterizability of $V$  should be interpreted  in the sense of the Montague-L\'evy kind of reflection. Rather, what it seems to suggest is some sort of reflection, not (only) of formulas, but of \emph{internal structural properties of the membership relation}. The quote does also state that the properties should be expressible in some logic, and any reasonable logic would do. So maybe G\"odel is not saying here anything new, and he is simply advocating for a generalizsation of the Montague-L\'evy type of reflection to formulas belonging to any (reasonable) kind of logic. But whatever the correct interpretation of G\"odel's quote above may be, let us consider in some detail the idea of reflection of structural properties of the membership relation. Thus, what one would want to reflect is not the theory of $V$, but rather the \emph{structural content} of $V$. 

Whatever one might mean by a ``structural property of the membership relation", it is clear that such a property  should be exemplified in structures of the form  $\langle A, \in , \langle R_i\rangle_{i\in I} \rangle$, where $A$ is a set   and $\langle R_i\rangle_{i\in I}$ is a family of relations on $A$, and where $I$ is a set that may be empty. Moreover, any such property should be expressible  by some formula of the language of set theory, maybe involving some set parameters. Thus, any internal structural property of the membership relation would be formally given by a formula $\varphi(x)$ of the first-order language of set theory, possibly with parameters, that defines a class of structures  $\langle A, \in ,\langle R_i\rangle_{i\in I}  \rangle$ of the same type.  As we shall see later on there is no loss of generality in considering only  classes of structures  whose members are of the form $\langle V_\alpha, \in ,\langle R_i\rangle_{i\in I}  \rangle$. Now, G\"odel's vague assertion (as quoted above) that $V$ ``cannot be uniquely characterized (i.e., distinguished from all its initial segments) by any internal structural property of the membership relation" can be naturally interpreted in the sense that no formula $\varphi (x)$ characterizes $V$, meaning that some  $V_\alpha$ reflects \emph{the structural property defined by $\varphi(x)$}. Let us  emphasize  that what is reflected is not the formula $\varphi(x)$, but the structural property defined by $\varphi(x)$, i.e., the class of structures defined by $\varphi(x)$. This is the crucial difference with the Montague-L\'evy type of reflection. The most natural way to make this precise is to assert  that there exists an ordinal $\alpha$  such that for every structure $A$  in the class  (i.e., for every structure $A$ that satisfies $\varphi(x)$) there exists a structure $B$ also in the class which belongs to $V_\alpha$ and is \emph{like} $A$. Since, in general,  $A$ may be much larger than any $B$ in $V_\alpha$, the closest resemblance of $B$ to $A$ is attained in the case $B$ is isomorphic to an elementary substructure of $A$, i.e., $B$ can be elementarily embedded into $A$.
We can now formulate (an informal and  preliminary version of) the  general principle of Structural Reflection as follows:

\begin{quote}
\begin{itemize}
\item[$\SR$:] (\emph{Structural Reflection}) For every definable, in the first-order language of set theory, possibly with parameters, class $\mathcal{C}$ of relational structures  of the same type there exists an ordinal $\alpha$ that \emph{reflects} $\mathcal{C}$, i.e.,  for every $A$ in $\mathcal{C}$ there exists $B$ in $\mathcal{C}\cap V_\alpha$ and an elementary embedding from $B$ into $A$.
\end{itemize}
\end{quote}
Observe that when $\mathcal{C}$ is a set the principle becomes trivial. 

We do not wish to claim that the $\SR$ principle is what G\"odel had in mind when talking about reflection of internal structural properties of the membership relation, but we do claim that $\SR$ is a form of reflection that derives naturally from the Cantor-Ackermann principle of unknowability of the Absolute and  is at least compatible with G\"odel's interpretation of this principle.

\medskip

In the remaining sections we will survey a collection of results showing the equivalence of different forms of $\SR$ with the existence of different kinds of large cardinals. Our goal is to illustrate the fact that $\SR$ is a general principle underlying a wide variety of large-cardinal principles. Many of the results have already been published (\cite{Ba:CC,  BCMR, BV,  BGS, BW, Lu:SR}) or are forthcoming (\cite{BL}), but some  are new (\ref{thmzerosharp}, \ref{thmzerosharp2}, \ref{remgsr}, \ref{thmzerosharp3}, \ref{thmzerodagger},
\ref{thmzerodagger2}, \ref{Sigma2}, \ref{mainstrong}, \ref{prop2}, \ref{mainsuperstrong2}, \ref{mainsuperstrong3}, \ref{PSRmeasurable}, and \ref{firstremarkable}). Each of these results should be regarded as a small step towards the ultimate objective of showing that all large cardinals are in fact different manifestations of a single general reflection principle.  

\section{From supercompactness to Vop\v{e}nka's Principle}

We shall begin with  the $\SR$ principle, as stated above, which is properly formulated in  the first-order language of set theory  as an axiom schema. Namely, for each natural number $n$ let

\begin{quotation}
\begin{itemize}
\item[$\mathbf{\Sigma_n}$-$\SR$:] (\emph{$\mathbf{\Sigma_n}$-Structural Reflection})    For every $\Sigma_n$-definable, with parameters, class $\mathcal{C}$ of relational structures of the same type there is an ordinal $\alpha$ that  \emph{reflects} $\mathcal{C}$. 
\end{itemize}
\end{quotation}
$\mathbf{\Pi_n}$-$\SR$ may be formulated  analogously. 
We may also define the \emph{lightface}, i.e., parameter-free versions (as customary we use the lightface types $\Sigma_n$ and $\Pi_n$ for that). Namely,

\begin{quotation}
\begin{itemize}
\item[$\Sigma_n$-$\SR$:] (\emph{$\Sigma_n$-Structural Reflection})   For every $\Sigma_n$-definable, without parameters, class $\mathcal{C}$ of relational structures of the same type there exists an ordinal $\alpha$ that  \emph{reflects} $\mathcal{C}$. 
\end{itemize}
\end{quotation}
Similarly for $\Pi_n$-$\SR$. 

A standard closing-off argument shows that $\mathbf{\Sigma_n}$-$\SR$ is equivalent to the assertion that there exists a proper class of ordinals  $\alpha$ such that $\alpha$ reflects \emph{all} classes of structures of the same type that are $\Sigma_n$-definable, with parameters in $V_\alpha$.   Also,  $\Sigma_n$-$\SR$ is equivalent to the assertion that there exists an ordinal  $\alpha$  that  reflects \emph{all} classes of structures of the same type that are $\Sigma_n$-definable, without parameters. Similarly for $\mathbf{\Pi_n}$-$\SR$ and $\Pi_n$-$\SR$. 
Thus, for $\Gamma$ a \emph{definability class} (i.e., one of $\mathbf{\Sigma_n}$, $\mathbf{\Pi_n}$, $\Sigma_n$, or $\Pi_n$), let us say that an ordinal $\alpha$ \emph{witnesses} $\Gamma$-$\SR$ if $\alpha$ reflects all classes of structures of the same type that are $\Gamma$-definable (allowing for  parameters in $V_\alpha$, in the case of boldface classes). Then, in the case of boldface classes $\Gamma$, $\Gamma$-$\SR$ holds if and only if $\Gamma$-$\SR$ is witnessed by a proper class of ordinals.



\medskip

The first observation is that, as the next proposition shows, $\mathbf{\Sigma_1}$-$\SR$ is provable in ZFC. Recall\footnote{See \cite{Ba:CC}.} that, for every $n>0$,  $C^{(n)}$ is the $\Pi_n$-definable club proper class of cardinals $\kappa$  such that $V_\kappa \preceq_{\Sigma_n} V$, i.e., $V_\kappa$ is a $\Sigma_n$-elementary substructure of $V$. In particular, every element of $C^{(1)}$ is an uncountable cardinal and a fixed point of the $\beth$ function. 

\begin{proposition}
\label{sigma1}
The following are equivalent for every ordinal $\alpha$:
\begin{enumerate}
\item $\alpha$ witnesses $\mathbf{\Sigma_0}$-$\SR$
\item $\alpha$ witnesses $\mathbf{\Sigma_1}$-$\SR$  
\item $\alpha \in C^{(1)}$.
\end{enumerate}
\end{proposition}

\begin{proof}
The implication (2)$\Rightarrow$(1) is trivial. The implication (3)$\Rightarrow$(2) is proved in \cite{BCMR}, using a  L\"owenheim-Skolem type of argument. 

To show that (1) implies (3), let $\alpha$ witness $\mathbf{\Sigma_0}$-$\SR$ and suppose $\varphi(a_1,\ldots ,a_n)$ is a $\Sigma_1$ sentence, with parameters $a_1,\ldots ,a_n$ in $V_\alpha$, that holds in $V$.
Let $\Ce$ be the $\Sigma_0$-definable, with $a_1,\ldots ,a_n$ as parameters, class of structures of the form $$\langle M,\in, \{a_1,\ldots,a_n\} \rangle$$ where $M$ is a transitive set that contains $a_1,\ldots ,a_n$. Let $M$ be any transitive set such that $M\models \varphi(a_1,\ldots ,a_n)$. By $\SR$, there is an elementary embedding
 $$j:\langle N,\in, \{a_1,\ldots,a_n\} \rangle\to \langle M,\in ,\{a_1,\ldots,a_n\}\rangle$$
 where $\langle N,\in \{a_1,\ldots,a_n\}\rangle\in \Ce \cap V_\alpha$.
 Since $j$ fixes $a_1,\ldots ,a_n$, by elementarity $N\models \varphi(a_1,\ldots,a_n)$, and by upwards absoluteness for $\Sigma_1$ sentences with respect to transitive sets, $V_\alpha\models \varphi(a_1,\ldots ,a_n)$. This shows $\alpha \in C^{(1)}$. 
\end{proof}

Thus, $\mathbf{\Sigma_1}$-$\SR$ is provable in ZFC, and therefore does not yield any large cardinals. But $\Pi_1$-$\SR$ does, and is indeed very strong. The following theorem  hinges on Magidor's characterization of the first supercompact cardinal as the first cardinal that reflects the $\Pi_1$-definable class of structures of the form $\langle V_\alpha, \in\rangle$, $\alpha$ an ordinal (\cite{Mag}).

\begin{theorem}[\cite{Ba:CC, BCMR}]
\label{thmpi1}
The following are equivalent:
\begin{enumerate}
\item $\Pi_1$-$\SR$.
\item $\Sigma_2$-$\SR$.
\item There exists a supercompact cardinal.
\end{enumerate}
\end{theorem}
The proof of the theorem shows in fact that the following are equivalent for an ordinal $\kappa$:
\begin{enumerate}
\item $\kappa$ is the least ordinal that witnesses $\SR$ for the $\Pi_1$-definable class of structures  $\langle V_\alpha ,\in\rangle$, $\alpha$ an ordinal.
\item $\kappa$ is the least cardinal that witnesses $\Pi_1$-$\SR$.
\item $\kappa$ is the least cardinal that witnesses $\mathbf{\Sigma_2}$-$\SR$.
\item $\kappa$ is the least supercompact cardinal.
\end{enumerate}
The following global parametrized version then follows. Namely,
\begin{theorem}[\cite{Ba:CC, BCMR}]
\label{thmpi1b}
The following are equivalent:
\begin{enumerate}
\item $\mathbf{\Pi_1}$-$\SR$.
\item $\mathbf{\Sigma_2}$-$\SR$.
\item There exists a proper class of supercompact cardinals.
\end{enumerate}
\end{theorem}
The proof of the theorem also shows  that if $\kappa$ witnesses  $\mathbf{\Pi_1}$-$\SR$, then $\kappa$ is either  supercompact  or a limit of supercompact cardinals.

\medskip

Some remarks are in order. First, the equivalence of  $\Pi_1$-$\SR$ and $\Sigma_2$-$\SR$, and also of their boldface forms,  is due to the following general fact.  Given a $\Sigma_{n+1}$ definable (possibly with parameters, and with $n>0$) class  $\mathcal{C}$ of relational structures of the same type, let $\mathcal{C}^\ast$ be the class of structures of the form $\langle V_\alpha , \in ,A\rangle$, where $\alpha$ is the least cardinal in $C^{(n)}$ such that  $A\in V_\alpha$ and $V_\alpha \models \varphi (A)$, where $\varphi (x)$ is a fixed $\Sigma_{n+1}$ formula that defines $\Ce$. Then,
$$A\in \mathcal{C}\mbox{  if and only if  }\langle V_\alpha , \in ,A\rangle \in \mathcal{C}^\ast.$$ 
Now notice that $\mathcal{C}^\ast$ is $\Pi_n$ definable, with the same parameters as $\mathcal{C}$ (see \cite{Ba:CC}). Moreover, if a cardinal $\kappa$ reflects the class $\Ce^\ast$, then it also reflects  $\Ce$: for if $A\in \Ce$, let $\alpha$ be the least cardinal in $C^{(n)}$ such that $\langle V_\alpha ,\in, A\rangle \models \varphi(A)$, where $\varphi (x)$ is a fixed $\Sigma_{n+1}$ formula that defines $\Ce$. Let $j:\langle V_\beta , \in , B\rangle\to \langle V_\alpha ,\in, A\rangle$ be elementary with $\langle V_\beta ,\in,B\rangle \in \Ce^\ast \cap V_\kappa$. Then, since $\beta\in C^{(n)}$ and $V_\beta \models \varphi(B)$, we have that $B\in \Ce$ and the restriction map   $j\restriction A:A\to B$ is an elementary embedding.  

\medskip

For $P$ a set or a proper class and $\Gamma$ a definability class, we shall write $\Gamma (P)$-$\SR$ for the assertion that $\SR$ holds for all $\Gamma$-definable, with parameters in $P$, classes of structure of the same type. Thus, e.g., $\mathbf{\Sigma_n}$-$\SR$ is $\Sigma_n(V)$-$\SR$, and $\Sigma_n$-$\SR$ is $\Sigma_n(\emptyset)$-$\SR$. Our remarks above yield now the following:

\begin{proposition}
\label{equiv}
For $P$ any set or proper class, the assertions $\Pi_n(P)$-$\SR$ and $\Sigma_{n+1}(P)$-$\SR$ are equivalent. In particular $\Pi_n$-$\SR$ and $\Sigma_{n+1}$-$\SR$ are equivalent; and so are $\mathbf{\Pi_n}$-$\SR$ and $\mathbf{\Sigma_{n+1}}$-$\SR$.
\end{proposition}

Second, the remarks above also show that for principles of Structural Reflection of the form $\Gamma(P)$-$\SR$ the relevant structures to consider are those of the form $\langle V_\alpha ,\in, A\rangle$, where $A\in V_\alpha$. Let us say that a structure is \emph{natural} if it is of this form. 
Therefore, we may reformulate $\Gamma$-$\SR$, for $\Gamma$ a lightface  definability class, as follows: 

\begin{quotation}
\begin{itemize}
\item[$\Gamma$-$\SR$:] (\emph{$\Gamma$-Structural Reflection. Second version})    There exists a cardinal $\kappa$ that  reflects all $\Gamma$-definable classes $\Ce$ of natural structures, i.e., for every $A\in \Ce$ there exists $B\in \Ce \cap V_\kappa$ and an elementary embedding $j:B\to A$.
\end{itemize}
\end{quotation}
The version for $\Gamma$ a boldface definability class being as follows:
\begin{quotation}
\begin{itemize}
\item[]   There exist a  proper class of cardinals $\kappa$ that  reflect all $\Gamma$-definable, with parameters in $V_\kappa$, classes $\Ce$ of natural structures. 
\end{itemize}
\end{quotation}
%
%

\medskip

At the next level of definitional complexity, i.e., $n=2$, we have the following:

\begin{theorem} [\cite{Ba:CC, BCMR}]
The following are equivalent:
\begin{enumerate}
\item $\Pi_2$-$\SR$.
\item There exists an extendible cardinal.
\end{enumerate}
\end{theorem}

The proof of the theorem shows that the first extendible cardinal is precisely the first cardinal that witnesses SR for one particular $\Pi_2$-definable class of natural structures.
The parameterized version also holds: 

\begin{theorem}[\cite{Ba:CC, BCMR}]
The following are equivalent:
\begin{enumerate}
\item $\mathbf{\Pi_2}$-$\SR$.
\item There exists a proper class of extendible cardinals.
\end{enumerate}
\end{theorem}

Moreover, if $\kappa$ witnesses  $\mathbf{\Pi_2}$-$\SR$, then $\kappa$ is either  extendible  or a limit of extendible cardinals.

For the higher levels of definitional complexity we need the notion of $C^{(n)}$-extendible cardinal from \cite{Ba:CC, BCMR}: 
$\kappa$ is \emph{$C^{(n)}$-extendible} if for every $\lambda$ greater than $\kappa$ there exists an elementary embedding $j:V_\lambda \to V_\mu$, some $\mu$, with $crit(j)=\kappa$, $j(\kappa) >\lambda$, and $j(\kappa)\in C^{(n)}$.  Note that the only difference with the notion of extendibility is that we require  the image of the critical point to be  in $C^{(n)}$. Also note that every extendible cardinal is $C^{(1)}$-extendible. We then have the following level-by-level characterizations of $\SR$ in terms of the existence of large cardinals:

\begin{theorem}[\cite{Ba:CC, BCMR}]
The following are equivalent for $n\geq 1$:
\begin{enumerate}
\item $\Pi_{n+1}$-$\SR$.
\item There exists a  $C^{(n)}$-extendible cardinal.
\end{enumerate}
\end{theorem}

\begin{theorem}[\cite{Ba:CC, BCMR}]
The following are equivalent for $n\geq 1$:
\begin{enumerate}
\item $\mathbf{\Pi_{n+1}}$-$\SR$.
\item There exists a proper class of $C^{(n)}$-extendible cardinals.
\end{enumerate}
\end{theorem}

Similarly as in the case of supercompact and extendible cardinals, the proofs of the theorems above actually show that the first $C^{(n)}$-extendible cardinal is  the first cardinal that witnesses SR for one single $\Pi_{n+1}$-definable class of natural structures. Also, if $\kappa$ witnesses  $\mathbf{\Pi_{n+1}}$-$\SR$, then $\kappa$ is either a  $C^{(n)}$-extendible  cardinal or a limit of $C^{(n)}$-extendible cardinals.

\medskip

Recall that \emph{Vop\v{e}nka's Principle $(\VP)$} is the assertion that for every proper class $\mathcal{C}$ of relational structures of the same type there exist $A\ne B$ in $\mathcal{C}$ such that $A$ is elementarily embeddable into $B$. In the first-order language of set theory $\VP$ can be formulated as a  schema. The following corollary to the theorems stated above yields a characterization of $\VP$ in terms of $\SR$. Moreover, it shows that, globally, the lightface and boldface forms of $\SR$ are equivalent.

\begin{theorem}[\cite{Ba:CC, BCMR}]
\label{thmVP}
The following schemata are equivalent:
\begin{enumerate}
\item $\SR$, i.e., $\mathbf{\Pi_n}$-$\SR$ for all $n$.
\item $\Pi_n$-$\SR$ for all $n$.
\item There exists a $C^{(n)}$-extendible cardinal, for every $n$.
\item There is a proper class of $C^{(n)}$-extendible cardinals, for every $n$.
\item $\VP$.
\end{enumerate}
\end{theorem}

\section{Structural Reflection below supercompactness}

We have just seen that a natural hierarchy of large cardinals in the region between the first supercompact cardinal and  $\VP$ can be characterized in terms of $\SR$. Now the question is if the same is true for other well-known regions of the large cardinal hierarchy. Since $\mathbf{\Sigma_1}$-SR is provable in ZFC and $\Pi_1$-SR implies already the existence of a  supercompact cardinal, if large cardinals weaker than supercompact admit a characterization as principles of structural reflection, then we need to look either for $\SR$ restricted to particular (collections of) $\Pi_1$-definable classes of  structures, or for classes of structures whose definitional complexity is between $\Sigma_1$ and $\Pi_1$ (e.g., $\Sigma_1$-definability with additional $\Pi_1$ predicates), or for weaker forms of structural reflection.
Let us consider first the $\SR$ principle   restricted to particular  definable  classes of structures contained in canonical inner models.

\subsection{Structural Reflection relative to canonical inner models}
\label{sectionL}
 
 There is one single class $\Ce$ of structures in $L$ that is $\Pi_1$-definable in $V$, without parameters, and such that $\SR (\Ce)$ is equivalent to the existence of $0^\sharp$. Namely, let $\Ce$ be the  class of structures of the form $\langle L_\beta ,\in ,\gamma\rangle$, with  
$\gamma <\beta$  uncountable cardinals (in $V$).
 
\begin{theorem}\label{thmzerosharp}
The following are equivalent:
\begin{enumerate}
\item $\SR(\Ce)$ 
%
%
\item $0^\sharp$ exists. 
\end{enumerate}
\end{theorem}

\begin{proof}
(1) implies (2): Suppose that $\alpha$ reflects $\Ce$. Pick $V$-cardinals $\gamma < \beta$ with $\alpha \leq \gamma$. By reflection, there are  $V$-cardinals $\gamma'<\beta'<\alpha$   and an elementary embedding 
$$j:\langle L_{\beta'},\in ,\gamma'\rangle \to \langle L_\beta ,\in ,\gamma\rangle.$$
Since $j(\gamma')=\gamma$, $j$ is not the identity. Let $\kappa$ be the critical point of $j$. Thus, $\kappa \leq \gamma'<\beta'$. Hence by Kunen's Theorem (\cite{Ku:EE}) $0^\sharp$ exists.

(2) implies (1): Assume $0^\sharp$ exists. Let $\alpha$ be an uncountable  limit cardinal in $V$. We claim that $\alpha$ reflects $\Ce$. For suppose  $\langle L_\beta , \in ,\gamma\rangle \in \Ce$ with $\alpha \leq \beta$. Let $\gamma'<\beta'<\alpha$ be uncountable cardinals in $V$ such that  $\gamma' \leq \gamma$. Let $I$ denote the class of Silver indiscernibles. Let $j:I\cap [\gamma',\beta']\to I\cap [\gamma ,\beta]$ be order-preserving and such that $j(\gamma')=\gamma$ and $j(\beta')=\beta$. Then $j$ generates  an elementary embedding $$j:\langle  L_{\beta'},\in ,\gamma'\rangle \to \langle L_\beta ,\in ,\gamma\rangle$$
as required.
\end{proof}

The existence of $0^\sharp$ yields also the $\SR$ principle restricted to classes of structures that are definable in $L$.

\begin{theorem}\label{thmzerosharp2}
If $0^\sharp$ exists, then $\SR(\Ce)$ holds for every class $\Ce$ that is definable in $L$, with parameters.
\end{theorem}

\begin{proof}
Fix $\mathcal{C}$ and a  formula $\varphi(x)$, possibly with ordinals $\alpha_0 <\ldots <\alpha_m$ as parameters,  that defines it in $L$. Let $\kappa$ be a limit of Silver indiscernibles greater than $\alpha_m$. We claim that $\kappa$ reflects $\mathcal{C}$. For  suppose  $B\in \mathcal{C}$. Without loss of generality, $B\not \in L_\kappa$. Since $0^\sharp$ exists, there is an increasing sequence of Silver indiscernibles $i_0,\ldots ,i_n, i_{n+1}$ and a formula $\psi(y, z_0,\ldots ,z_n)$, without parameters, such that 
$$B=\{ y: L_{i_{n+1}}\models \psi (y, i_0,\ldots ,i_n)\}.$$
Choose indiscernibles $j_0 <\ldots < j_n <j_{n+1} <\kappa$, with $\alpha_m < j_0$,  and let
$$A=\{ y: L_{j_{n+1}} \models \psi (y,j_0,\ldots ,j_n)\}.$$
Thus $A\in L_\kappa$. 
Since $L\models \varphi(B)$, we have that
$$L\models \forall x(\forall y (y\in x\leftrightarrow L_{i_{n+1}}\models \psi (y, i_0,\ldots ,i_n))\to \varphi (x)).$$
By indiscernibility, 
$$L\models \forall x(\forall y (y\in x\leftrightarrow L_{j_{n+1}}\models \psi (y, j_0,\ldots ,j_n))\to \varphi (x))$$
which implies $L\models \varphi (A)$, i.e., $A\in \mathcal{C}$. 

Let $j:L\to L$ be an  elementary embedding that  sends $i_k$ to $j_k$, all $k\leq n+1$. Then by indiscernibility, the map $j\restriction A:A\to B$ is an elementary embedding.
\end{proof}

However, the $\SR$ principle restricted to classes of structures that are definable in $L$ falls very short of yielding $0^\sharp$, as we shall next show. 
Let us recall the following definition: 

\begin{definition}[\cite{BGS}]
\label{n-rem}
A cardinal $\kappa$ is \emph{$n$-remarkable}, for $n>0$, if for all $\lambda >\kappa$ in $C^{(n)}$ and every $a\in V_\lambda$, there is $\bar{\lambda} <\kappa$ also in $C^{(n)}$ such that in $V^{\rm{Coll}(\omega , <\kappa)}$ there exists an elementary embedding $j:V_{\bar{\lambda}}\to V_{\lambda}$ with $j({\rm{crit}}(j))=\kappa$ and $a\in {\rm{range}}(j)$.
\end{definition}

A cardinal $\kappa$ is $1$-remarkable if and only if it is remarkable, in the sense of Schindler (see \cite{BGS} and definition \ref{defrem} below).

If $0^\sharp$ exists, then every Silver indiscernible is completely remarkable in $L$  (i.e., $n$-remarkable for every $n>0$). Moreover, the consistency strength of the existence of a $1$-remarkable cardinal is strictly weaker than the existence of a $2$-iterable cardinal, which in turn is weaker than the existence of an $\omega$-Erd\"os cardinal (see \cite{BGS}).

A weaker notion than $n$-remarkability is obtained by eliminating from its definition the requirement that $j(\crit(j))=\kappa$. So, let's define:

\begin{definition}
\label{defarem}
A cardinal $\kappa$ is \emph{almost $n$-remarkable}, for $n>0$, if for all $\lambda >\kappa$ in $C^{(n)}$ and every $a\in V_\lambda$, there is $\bar{\lambda} <\kappa$ also in $C^{(n)}$ such that in $V^{\rm{Coll}(\omega , <\kappa)}$ there exists an elementary embedding $j:V_{\bar{\lambda}}\to V_{\lambda}$ with  $a\in {\rm{range}}(j)$.

We say that  $\kappa$ is \emph{almost completely-remarkable} if it is almost-$n$-remarkable for every $n$.
\end{definition}

\begin{theorem}
\label{remgsr}
A cardinal $\kappa$ is almost $n$-remarkable if and only if in $V^{\rm{Coll}(\omega,<\kappa)}$ $\kappa$ witnesses $\SR(\Ce)$  for every class $\Ce$ that is $\Pi_n$-definable in $V$ with parameters in $V_\kappa$.
\end{theorem}

\begin{proof}
Assume  $\kappa$ is almost $n$-remarkable. Fix $\mathcal{C}$ and a  $\Pi_n$ formula $\varphi(x)$ that defines it in $V$, possibly with parameters in $V_\kappa$.  Suppose  $B\in \mathcal{C}$.  In $V$, let $\lambda\in C^{(n)}$ be  greater than the rank of $B$.  Thus, $V_\lambda \models  \varphi (B)$. Since $\kappa$ is almost $n$-remarkable, there is $\bar{\lambda} <\kappa$ also in $C^{(n)}$ such that in $V^{\rm{Coll}(\omega,<\kappa)}$ there exists an elementary embedding $j:V_{\bar{\lambda}}\to V_{\lambda}$ with $B \in {\rm{range}}(j)$. Let $A$ be the preimage of $B$ under $j$. 
So $A\in V_\kappa$. 
By elementarity of $j$,  
$V_{\bar{\lambda}}\models  \varphi (A)$. 
Hence, since $\bar{\lambda} \in C^{(n)}$, $A\in \mathcal{C}$.
Moreover, $j\restriction A: A\to B$  is an elementary embedding.

Conversely, assume that in $V^{\rm{Coll}(\omega,<\kappa)}$, $\kappa$ witnesses $\SR(\Ce)$ for every class $\Ce$ that is $\Pi_n$-definable in $V$ with parameters in $V_\kappa$. Let $\Ce$ be the  $\Pi_n$-definable class of structures of the form $\langle V_\alpha ,\in, a\rangle$ where $\alpha \in C^{(n)}$ and $a\in V_\alpha$. Given $\lambda \in C^{(n)}$ and $a\in V_\lambda$, in $V^{\rm{Coll}(\omega,<\kappa)}$ there exists some  $\langle V_{\bar{\lambda}}, \in ,b\rangle \in \Ce$ together with an elementary embedding $$j:\langle V_{\bar{\lambda}}, \in, b \rangle \to \langle V_\lambda ,\in ,a\rangle.$$ Since $j(b)=a$, $a\in {\rm{range}}(j)$. This shows $\kappa$ is almost $n$-remarkable. 
\end{proof}

\begin{corollary}
If $\kappa$ is an almost completely-remarkable cardinal in $L$,  then in $L^{\rm{Coll}(\omega,<\kappa)}$ $\kappa$ witnesses $\SR(\Ce)$ for all classes $\Ce$ of structures that are definable in $L$ with parameters in $L_\kappa$.
\end{corollary}

Similar arguments yield analogous results for $X^\sharp$, for every set of ordinals $X$.  
Given a set of ordinals $X$, let $\mathcal{C}_X$ be the class of structures of the form $\langle L_\beta [X] ,\in ,\gamma\rangle$, where  $\gamma$ and $\beta$  are cardinals (in $V$) and  $sup(X) <\gamma <\beta$. Clearly, $\Ce_X$ is $\Pi_1$ definable with $X$ as a parameter. 

\begin{theorem}\label{thmzerosharp3}
$ $
\begin{enumerate}
\item $\SR(\Ce_X)$ holds if and only if $X^\sharp$ exists.
\item If $X^\sharp$ exists, then $\SR(\mathcal{C})$ holds for all classes  $\mathcal{C}$  that are definable  in $L[X]$, with parameters.
\end{enumerate}
\end{theorem}

These results suggest the following forms of $\SR$ restricted to  inner models. Let $M$ be an inner model. Writing $M_\alpha$ for $(V_\alpha)^M$, consider the following principle for $\Gamma$ a lightface definability class:

\begin{quotation}
\begin{itemize}
\item[$\Gamma$-$\SR(M)$:] (\emph{$\Gamma$-Structural Reflection for $M$})   There exists an ordinal $\alpha$ that  reflects every $\Gamma$-definable class $\Ce$ of relational structures of the same type such that $\Ce \subseteq M$, i.e.,    for every $A$ in $\mathcal{C}$ there exists $B$ in $\mathcal{C}\cap M_\alpha$ and an elementary embedding $j$ from $B$ into $A$. (\emph{Warning:} $j$ may not be in $M$.)
\end{itemize}
\end{quotation}
The corresponding version for a boldface $\Gamma$  being as follows:
\begin{quotation}
\begin{itemize}
\item[]   There exist a proper class of ordinals $\alpha$ that  reflect every $\Gamma$-definable, with parameters in $M_\alpha$, class $\Ce$ of relational structures of the same type such that $\Ce \subseteq M$. \end{itemize}
\end{quotation}



The last theorem shows that, for every set of ordinals $X$, the existence of $X^\sharp$ implies $\SR(L[X])$, i.e., $\mathbf{\Sigma_n}$-$\SR(L[X])$, for every $n$.

Similar results may be obtained for other canonical inner models, e.g., $L[U]$, the canonical inner model for one measurable cardinal $\kappa$, where $U$ is the (unique) normal measure on $\kappa$ in $L[U]$.  Let $\Ce^U$ be the  class of structures of the form $\langle L_\beta[U] ,\in ,\gamma\rangle$, with  
$\gamma <\beta$  uncountable cardinals (in $V$). The class $\Ce^U$ is $\Pi_1$-definable in $V$,  with $U$ as a parameter. 
By using well-known facts about $0^\dagger$ due to Solovay (see \cite{Kan:THI} 21), and arguing similarly as in Theorems \ref{thmzerosharp} and  \ref{thmzerosharp2}, respectively, one obtains the following:

\begin{theorem}
\label{thmzerodagger}
The following are equivalent:
\begin{enumerate}
\item $\SR(\Ce^U)$ holds if and only if $0^\dagger$ exists. 

\item 
If $0^\dagger$ exists, then $\SR(\Ce)$ holds for every class $\Ce$ that is definable in $L[U]$, with parameters.
\end{enumerate}
\end{theorem}

Also, similarly as in Theorem \ref{thmzerosharp3},  one can obtain the analogous result for $X^\dagger$, for every set of ordinals $X$.  Namely, 
given a set of ordinals $X$, let $\mathcal{C}^U_X$ be the class of structures of the form $\langle L_\beta [U, X] ,\in ,\gamma\rangle$, where  $\gamma$ and $\beta$  are cardinals (in $V$) and  $sup(X) <\gamma <\beta$. Then $\Ce^U_X$ is $\Pi_1$-definable with $U$ and $X$ as  parameters. 

\begin{theorem}\label{thmzerodagger2}
$ $
\begin{enumerate}
\item $\SR(\Ce^U_X)$ holds if and only if $X^\dagger$ exists.
\item If $X^\dagger$ exists, then $\SR(\mathcal{C})$ holds for all classes  $\mathcal{C}$  that are definable  in $L[U,X]$, with parameters.
\end{enumerate}
\end{theorem}

Analogous results should also hold for canonical inner models for stronger large-cardinal notions. E.g., for the canonical inner model $L[\mathcal{E}]$  for a strong cardinal, as in \cite{Ko:Ext} or \cite{Mitchell2010}, and its sharp, \emph{zero pistol} $0^\P$. Also for a canonical inner model for a proper class of strong cardinals, as in \cite{Sch:Core}, and its sharp, \emph{zero hand grenade}. For inner models for stronger large cardinal notions, e.g., one Woodin cardinal, the situation is less clear, although analogous results should hold given the appropriate canonical inner model and its corresponding sharp.

\section{Product Structural reflection}

Recall that for any set $S$ of relational structures $\mathcal{A}=\langle A, \ldots \rangle$ of the same type,  the set-theoretic product $\prod S$  is the structure whose universe is the  set of all functions $f$ with domain $S$ such that $f(\mathcal{A})\in A$, for every $\mathcal{A}\in S$, and whose relations are defined point-wise.

In this section we shall consider the following general product form of structural reflection, which is a variation of the Product Reflection Principle $(\PRP)$ introduced in \cite{BW}:

\begin{quotation}
\begin{itemize}
\item[$\PSR$:] (\emph{Product Structural Reflection}) For every definable class of relational structures $\mathcal{C}$ of the same type, $\tau$, there exists an ordinal $\alpha$ that   {\emph{product-reflects}} $\mathcal{C}$, i.e.,  
for every $\Ae$ in $\mathcal{C}$ there exists a set $S$ of structures of type $\tau$ (although not necessarily in $\Ce$) with $\Ae\in S $ and an elementary embedding $j:\prod (\Ce \cap V_\alpha)\to \prod S$.
\end{itemize}
\end{quotation}

Similarly as in the case of $\SR$ (see section \ref{SR}), we may formally define $\PSR$  as a  schema. 
Thus, we say that an ordinal $\alpha$ witnesses $\Gamma(P)$-$\PSR$ (where $\Gamma$ is a    definability  class and $P$ a set or a proper class) if $\alpha$ product-reflects all classes $\Ce$ that are $\Gamma$-definable with parameters in $P$. 
Our remarks in section \ref{SR} also apply here. In particular, an ordinal $\alpha$ witnesses $\mathbf{\Pi_n}$-$\PSR$ if and only if it witnesses $\mathbf{\Sigma_{n+1}}$-$\PSR$. Moreover, we obtain equivalent principles by restricting to classes of natural structures. Thus, for $\Gamma$ a lightface definability class, we define:
\begin{quotation}
\begin{itemize}
\item[$\Gamma$-$\PSR$:] (\emph{$\Gamma$-Product Structural Reflection})    There exists an  ordinal $\alpha$ that   product-reflects all $\Gamma$-definable  classes $\Ce$ of natural structures. 
\end{itemize}
\end{quotation}
The corresponding  version for $\Gamma$ boldface being as follows:
\begin{quotation}
\begin{itemize}
\item[]    There exist a proper class of ordinals $\alpha$ that   product-reflect all $\Gamma$-definable, with parameters in $V_\alpha$, classes $\Ce$ of natural structures. 
\end{itemize}
\end{quotation}

As in  \cite[Proposition 3.2]{BW} one can show that  every cardinal $\kappa$ in $C^{(1)}$ witnesses $\mathbf{\Sigma_1}$-$\PSR$. The converse also holds, and in fact we have the following:

\begin{proposition}
\label{Sigma2}
For every $n$, if $\kappa$ witnesses $\mathbf{\Pi_n}$-$\PSR$, then $\kappa \in C^{(n+1)}$. 
\end{proposition}

\begin{proof}
We shall prove the case $n=1$. The general case follows by induction, using a similar argument. The case $n=0$ is similar to the case $n=1$, but simpler, as it suffices  to consider a class of structures with domain a transitive set (see the proof of Proposition \ref{sigma1}). So, suppose $\varphi(x,y)$ is a $\Pi_1$ formula with $x,y$ as the only free variables, $a\in V_\kappa$, and $V\models \exists x\varphi(x,a)$. Let $\Ce$ be the $\Pi_1$-definable, with $a$ as a parameter,  class of structures of the form $\Ae_\alpha = \langle V_{\alpha }, \in , a ,\{ R^{\alpha}_\varphi \}_{\varphi \in \Pi_1} \rangle$ with $\alpha \in C^{(1)}$, and where $\{ R^{\alpha}_\varphi\}_{\varphi \in \Pi_1}$ is the $\Pi_1$ relational diagram for $\langle V_{\alpha}, \in ,a\rangle$, i.e., if $\varphi(x_1,\ldots ,x_n, \overline{a})$, $\overline{a}$ a constant symbol, is a $\Pi_1$ formula in the language of $\langle V_{\alpha},\in ,a \rangle$, then $$R^\alpha_\varphi =\{ \langle a_1,\ldots ,a_n\rangle: \langle V_{\alpha},\in,a\rangle\models ``\varphi[a_1,\ldots ,a_n, a]"\}\, .$$  Let $\lambda >\kappa$ be in $C^{(2)}$, so that $V_\lambda \models \exists x\varphi(x,a)$. By $\PSR$  there exists a set $S$ that contains $\Ae_\lambda$ and an elementary embedding 
$$j:\prod_{\alpha <\kappa} \Ae_\alpha \to \prod S.$$

Since $\Ae_\lambda \models \exists x \varphi(x,a)$, we have $R^\lambda_\varphi  \ne \emptyset$. Hence,  since $\Ae_\lambda \in S$,
$$\prod S \not\models \langle R^\alpha_\varphi\rangle_{\Ae_\alpha \in S}  = \emptyset$$
and therefore, by elementarity of $j$, 
$$\prod_{\alpha <\kappa}\Ae_\alpha \not\models \langle R^\alpha_\varphi\rangle_{\alpha <\kappa}  = \emptyset$$
which implies that $\Ae_\alpha \models R^\alpha_\varphi \ne \emptyset$, for some $\alpha <\kappa$.
Hence, $\Ae_\alpha \models \exists x \varphi(x,a)$. 

Note that, if $\varphi(x,y)$ had been a bounded formula, instead of $\Pi_1$, then we would have, by upward absoluteness, that $\Ae_\kappa \models \exists x\varphi(x,y)$,   thus showing  that $\kappa \in C^{(1)}$. 
Thus, since $\alpha , \kappa\in C^{(1)}$, we have that $V_\alpha \preceq_{\Sigma_1}V_\kappa$, and therefore
$V_\kappa \models \exists x\varphi(x,a)$.

Now suppose $V_\kappa \models \exists x\varphi(x,a)$. Since $\kappa\in C^{(1)}$, by upward absoluteness, $V\models \exists x\varphi(x,a)$.
\end{proof}

Recall  that a cardinal $\kappa$ is \emph{$\lambda$-strong}, where $\lambda>\kappa$,  if there exists an elementary embedding $j: V\to M$, with $M$ transitive, $\crit(j)=\kappa$,  $j(\kappa)>\lambda$, and with $V_\lambda$ contained in $M$. A cardinal $\kappa$ is \emph{strong} if it is $\lambda$-strong for every cardinal $\lambda >\kappa$.

The  following  proposition is proved similarly as in \cite[3.3]{BW}.

\begin{proposition}
\label{prop1}
If $\kappa$ is a strong cardinal, then $\kappa$ witnesses $\mathbf{\Pi_1}$-$\PSR$.
\end{proposition}

\begin{proof}
Let $\kappa$ be a strong cardinal and let $\Ce$ be a $\Pi_1$-definable, with parameters in $V_\kappa$,  proper class of structures in a fixed  relational language $\tau\in V_\kappa$. Let $\varphi (x)$ be a $\Pi_1$ formula defining it.  

Given any $\mathcal{A}\in \Ce$, let $\lambda\in C^{(1)}$ be greater than or equal to $\kappa$ and with  $\mathcal{A}\in V_\lambda$.

Let $j:V\to M$ be an elementary embedding, with $crit(j)=\kappa$, $V_\lambda \subseteq M$, and $j(\kappa)>\lambda$.

By elementarity, the restriction of $j$ to $\Ce \cap V_\kappa$ yields an elementary embedding
$$h:\prod (\Ce \cap V_\kappa)\to \prod (\{ X: M\models \varphi(X)\} \cap V^M_{j(\kappa)}).$$
Let $S:= \{ X: M\models \varphi(X)\} \cap V^M_{j(\kappa)}$. Since $\mathcal{A}\in V_\lambda$ and $\varphi(x)$ is $\Pi_1$, by downward absoluteness $V_\lambda \models \varphi(\mathcal{A})$. Hence, since the fact that $\lambda \in C^{(1)}$ is $\Pi_1$-expressible and therefore downwards absolute for transitive classes, and since $V_\lambda \subseteq M$, it follows  that  $V_\lambda \preceq_{\Sigma_1}M$ and therefore  $M\models \varphi(\mathcal{A})$. Moreover $\mathcal{A}\in V_\lambda \subseteq V^M_{j(\kappa)}$. Thus, $\Ae \in S$. 
\end{proof}

Since the Product Reflection Principle ($\PRP$) introduced in \cite{BW} when  restricted to $\Pi_1$-definable classes is an easy consequence of $\Pi_1$-$\PSR$, from \cite{BW} we obtain the following: 

\begin{theorem}
The following are equivalent:
\begin{enumerate}
\item $\Pi_1$-$\PSR$
\item There exists a strong cardinal.
\end{enumerate}
\end{theorem}
as well as its  boldface version:
\begin{theorem}
The following are equivalent:
\begin{enumerate}
\item $\mathbf{\Pi_1}$-$\PSR$
\item There exists a proper class of strong cardinals.
\end{enumerate}
\end{theorem}

This shows that strong cardinals are related to $\PSR$ as supercompact cardinals are to $\SR$ (Theorems \ref{thmpi1} and \ref{thmpi1b}). For the higher levels of definability, i.e., $n>1$, the large cardinal notion that corresponds to $\Pi_n$-$\PSR$, analogous to the notion of $C^{(n)}$-extendible cardinal in the case of $\SR$, is the  following: 

\begin{definition} \cite{BW}
\label{defSigmaStrong}
For $\Gamma$ a definability class,  a cardinal $\kappa$  is \emph{$\lambda$-$\Gamma$-strong}, $\lambda$ an ordinal,   if for every $\Gamma$-definable (without parameters)  class $A$ there is   an elementary embedding $j:V\to M$, with $M$ transitive, $\crit(j)=\kappa$, $V_\lambda \subseteq M$,  and $A\cap V_\lambda \subseteq j(A)$.

A cardinal $\kappa$ is \emph{$\Gamma$-strong} if it is $\lambda$-$\Gamma$-strong for every ordinal $\lambda$.


\end{definition}

As with the case of strong cardinals, standard arguments show (cf. \cite{Kan:THI} 26.7(b)) that $\kappa$  is \emph{$\lambda$-$\Gamma$-strong} if and only if for every $\Gamma$-definable (without parameters)  class $A$ there is   an elementary embedding $j:V\to M$, with $M$ transitive, $\crit(j)=\kappa$, $V_\lambda \subseteq M$,  $j(\kappa)>\lambda$, and $A\cap V_\lambda \subseteq j(A)$. As shown in \cite{BW}, every strong cardinal   is $\Sigma_2$-strong. Also, a cardinal is $\Pi_n$-strong if and only if is $\Sigma_{n+1}$-strong. Moreover, if $n\geq 1$ and $\lambda\in C^{(n+1)}$, then the following are equivalent for a cardinal $\kappa <\lambda$:
\begin{enumerate}
\item $\kappa$ is $\lambda$-$\Pi_{n}$-strong. 
\item There is an elementary embedding $j:V\to M$, with $M$ transitive, $\crit(j)=\kappa$, $V_\lambda \subseteq M$,  and  $M\models ``\lambda \in C^{(n)}"$.
\end{enumerate}

Similarly as in Proposition \ref{prop1}, one can prove the following (see \cite{BW} for details):
\begin{prop}
\label{prop6}
If $\kappa$ is a $\Pi_n$-strong cardinal,  then $\kappa$ witnesses $\mathbf{\Pi_{n}}$-$\PSR$.
\end{prop}
The following theorem then follows from the main result in \cite{BW}:
\begin{theorem}
The following are equivalent for $n\geq 1$:
\begin{enumerate}
\item $\Pi_{n}$-$\PSR$
\item There exists a $\Pi_{n}$-strong cardinal.
\end{enumerate}
\end{theorem}
The corresponding boldface version  also holds. Namely,
\begin{theorem}
The following are equivalent for $n\geq 1$:
\begin{enumerate}
\item $\mathbf{\Pi_{n}}$-$\PSR$
\item There exists a proper class of $\Pi_{n}$-strong cardinals.
\end{enumerate}
\end{theorem}

Finally, for the $\PSR$ principle, the statement analogous to Vop\v{e}nka's Principle in the case of $\SR$ is the following:

\begin{definition}\cite{BW}
\emph{$\OR$ is Woodin} if for every definable $A\subseteq V$ there exists some $\alpha$ which is $A$-strong, i.e., for every $\gamma$ there is an elementary embedding $j:V\to M$ with $\crit (j)=\alpha$, $\gamma <j(\alpha)$, $V_\gamma \subseteq M$, and  $A\cap V_\gamma =j(A)\cap V_\gamma$.
\end{definition}

Note that if $\delta$ is a Woodin cardinal (see \cite{Kan:THI} for the definition of Woodin cardinal and its equivalent formulation in terms of $A$-strength), then $V_\delta$ satisfies $\OR$ is Woodin. The following equivalences then follow (cf.$ $ Theorem \ref{thmVP}):

\begin{theorem}[\cite{BW}]
The following schemata are equivalent:
\begin{enumerate}
\item $\PSR$, i.e., $\mathbf{\Pi_n}$-$\PSR$ for all $n$.
\item $\Pi_n$-$\PSR$ for all $n$.
\item There exists a $\Pi_n$-strong cardinal, for every $n$.
\item There is a proper class of $\Pi_n$-strong cardinals, for every $n$.
\item $\OR$ is Woodin.
\end{enumerate}
\end{theorem}

A close inspection of the proofs of Propositions \ref{prop1} and \ref{prop6} reveals that, for $n>0$,  if $\kappa$ is a $\Pi_n$-strong cardinal, then for every $\Pi_n$-definable (with parameters in $V_\kappa$)  class $\Ce$ of relational  structures of the same type $\tau$, and for every $\beta\geq \kappa$, there exists a set $S$ of structures of type $\tau$ (although possibly not in $\Ce$) that contains $\Ce \cap V_\beta$ and there exists an elementary embedding  $h:\prod (\Ce \cap V_\kappa)\to \prod S$ with the following  properties:
\begin{enumerate}
\item  \emph{Faithful}: For every $f\in \prod (\Ce \cap V_\kappa)$,  $h(f)\restriction (\Ce \cap V_\kappa) =f$.
\item \emph{$\subseteq$-chain-preserving}: If $f\in \prod (\Ce \cap V_\kappa)$ is so that $f(\Ae)\subseteq  f(\Ae')$ whenever $A \subseteq A'$, then so is $h(f)$. 
\end{enumerate}
Moreover, if $\kappa$ witnesses $\Pi_n$-$\PSR$, then some cardinal les than or equal to $\kappa$ is $\Pi_n$-strong.
Thus, the following is an equivalent reformulation of   $\Gamma$-$\PSR$, for $\Gamma=\Gamma_n$ a lightface definability class with $n>0$:
\begin{quotation}
\begin{itemize}
\item[$\Gamma$-$\PSR$:] (\emph{$\Gamma$-Product Structural Reflection. Second version}) There exists a   cardinal $\kappa$ that \emph{product-reflects} all $\Gamma$-definable proper classes  $\Ce$ of relational structures of the same type $\tau$, i.e.,  for every $\beta$ there exists a set $S$ of structures of type $\tau$ that contains $\Ce \cap V_\beta$, and there exists  a faithful and  $\subseteq$-chain-preserving elementary embedding  $h:\prod (\Ce \cap V_\kappa)\to \prod S$.
\end{itemize}
\end{quotation}
The corresponding version for a boldface $\Gamma$ being as follows:
\begin{quotation}
\begin{itemize}
\item[] There exist a proper class of cardinals $\kappa$ that \emph{product-reflect} all $\Gamma$-definable, with parameters in $V_\kappa$, proper class $\Ce$ of relational structures of the same type. 
\end{itemize}
\end{quotation}

The next theorem implies that strong cardinals can be characterized in terms of $\Pi_1$-$\PSR$. The proof follows closely \cite[Theorem 5.1]{BW}, with the properties of faithfulness and $\subseteq$-chain preservation ((1) and (2) above) playing now a key role.

\begin{theorem}
\label{mainstrong}
There is a $\Pi_1$-definable, without parameters, class $\Ce$ of natural structures such that if a cardinal  $\kappa$ product-reflects $\Ce$ (second version), then $\kappa$ is a strong cardinal.
\end{theorem}

\begin{proof}
Let $C$ be the  class of all ordinals $\alpha <\kappa$ of uncountable cofinality such that $\alpha$ is the $\alpha$-th element of $C^{(1)}$. Let $\mathcal{C}$ be the $\Pi_1$-definable class of all structures  
$$\Ae_\alpha:=\langle V_{\lambda_\alpha}, \in ,  \alpha  \rangle$$  where  $\alpha\in C$, 
and $\lambda_\alpha$ is the least  cardinal in $C^{(1)}$ greater than $\alpha$

Note that, since by proposition \ref{Sigma2},  $\kappa\in C^{(2)}$,  $\kappa$ is a limit point of $C$.

Pick any $\gamma$  in $C$ greater than $\kappa$. We will show that $\kappa$ is $\gamma$-strong. By $\PSR$ there is  a faithful $\subseteq$-chain-preserving  elementary embedding $j:\prod (\Ce \cap V_\kappa)  \to \prod S$, where $S$ is some set with $(\Ce \cap V_\kappa)  \cup \{\Ae_\gamma\}\subseteq  S$. 

\medskip

Now pick any $\Ae_\beta \in \Ce \cap S$ and let  $$h_\beta :\prod S\to \Ae_\beta $$ be the projection map.
Let $I:=C\cap \kappa$ and define
$$k_\beta:V_{\kappa +1}\to V_{\beta+1}$$
by:
$$k_\beta(X)=h_\beta (  j(\{ X\cap V_{\alpha}\}_{\alpha \in I})).$$ Since $j$ is elementary, for all formulas $\varphi (x_1,\ldots ,x_n)$ and all $a_1,\ldots ,a_n \in \prod(\Ce \cap V_\kappa)$, if $\prod(\Ce \cap V_\kappa)\models \varphi [a_1,\ldots ,a_n]$, then $\prod S \models \varphi [j(a_1)\ldots ,j(a_n)].$
It easily follows that $k_\beta$ preserves Boolean operations, the subset relation, and is the identity on $\omega +1$. 

Note that $k_\beta(\kappa)=h_\beta( j(\{ \alpha\}_{\alpha \in I}))= \beta.$

\medskip

For each $a\in [\beta]^{<\omega}$,   define $E^\beta_a$ by
$$X\in E^\beta_a \quad \mbox{ iff }\quad X\subseteq [\kappa]^{|a|}  \mbox{ and } a\in k_\beta(X)\, .$$
Since $k_\beta(\kappa)=\beta$ and $k_\beta(|a|)=|a|$, we also have $k_\beta([\kappa]^{|a|})=[\beta]^{|a|}$, hence $[\kappa]^{|a|} \in E^\beta_a$. Since $k_\beta$ preserves Boolean operations and the $\subseteq$ relation, $E^\beta_a$ is a proper ultrafilter over $[\kappa]^{|a|}$. Moreover, since $k_\beta(\omega)=\omega$, a simple argument shows that $E^\beta_a$ is $\omega_1$-complete, hence the ultrapower ${\rm{Ult}}(V,E^\beta_a)$  is well-founded. Furthermore, since $j$ is faithful, if $\beta <\kappa$, then $E^\beta_a$ is the principal ultrafilter generated by $\{ a\}$. Let 
$$j^\beta_a:V\to M^\beta_a \cong {\rm{Ult}}(V, E^\beta_a)$$
with $M_a^\beta$ transitive, be the corresponding ultrapower embedding. Note that if $\beta  <\kappa$, then $M^\beta_a=V$ and $j^\beta_a$ is the identity.

\medskip

Let $\mathcal{E}_\beta:=\{ E^\beta_a: a\in [\beta]^{<\omega}\}$.  As in \cite{BW}, one can show that $\mathcal{E}_\beta$ is normal and coherent.
Thus, for each $a\subseteq b$ in $[\beta]^{<\omega}$ the maps $i^\beta_{ab}:M^\beta_a\to M^\beta_b$   given by 
$$i^\beta_{ab}([f]_{E^\beta_a})=[f\circ \pi_{ba}]_{E^\beta_b}$$
for all $f:[\kappa]^{|a|}\to V$,  are  well-defined and commute with the ultrapower embeddings $j^\beta_a$ (see \cite{Kan:THI} 26). 

Let   $M_{\Ee_\beta}$ be  the  direct limit of the directed system $$\langle \langle M^\beta_a:a\in [\beta]^{<\omega}\rangle, \langle i^\beta_{ab}:a\subseteq b\rangle\rangle$$   
and let $j_{\Ee_\beta} :V\to M_{\Ee_\beta}$ be the corresponding direct limit elementary embedding, i.e., 
$$j_{\Ee_\beta}(x)=[a,[c^a_x]_{E^\beta_a}]_{\Ee_\beta}$$
for some (any) $a\in [\beta]^{<\omega}$, and  where  $c^a_x:[\kappa]^{|a|}\to \{ x\}$. 


As in \cite{BW} one can also show  that
$M_{\Ee_\beta}$ is  well-founded. 
So, let $\pi_\beta:M_{\Ee_\beta} \to N_\beta$ be the transitive collapse, and let $j_{N_\beta}:V\to N_\beta$ be the corresponding elementary embedding, i.e., $j_{N_\beta}=\pi \circ j_{\Ee_\beta}$. Then, as in \cite{BW} we can show  that $V_\beta \subseteq N_\beta$ and 
$j_{N_\beta}(\kappa)\geq \beta$.
If $\beta >\kappa$, this implies that $\crit(j_{N_\beta})\leq \kappa$. (If $\beta <\kappa$, then $j_{N_\beta}:V\to V$ is the identity.)

\medskip

Let $I_S:=\{ \beta : \Ae_\beta \in \Ce \cap S\}$.

\begin{claim}
\label{claimcoherent}
If $\beta \leq \beta'$ are in $I_S$, then $E^\beta_a = E^{\beta'}_a$, for every $a\in [\beta]^{<\omega}$.
\end{claim}

\begin{proof}[Proof of claim]
Since $E^\beta_a, E^{\beta'}_a$ are proper ultrafilters over $[\kappa]^{|a|}$, it is sufficient to see that $E^\beta_a\subseteq E^{\beta'}_a$. So, suppose $X\in E^\beta_a$. Then $a\in k_\beta(X)=h_\beta (j(\{ X\cap V_\alpha\}_{\alpha \in I}))$. Since $\{ X\cap V_\alpha\}_{\alpha \in I}$ forms an $\subseteq$-chain and $j$ is $\subseteq$-chain-preserving, so does $j(\{ X\cap V_\alpha\}_{\alpha \in I})$. Hence, $k_\beta (X)\subseteq k_{\beta'}(X)$, and therefore $a\in k_{\beta'}(X)$, which yields $X\in E^{\beta'}_a$.
\end{proof}

By the claim above, for every $\beta <\beta'$ in $I_S$ the map
$$k_{\beta ,\beta'}:M_{\Ee_{\beta}}\to M_{\Ee_{\beta'}}$$
given by
$$k_{\beta ,\beta'}([a, [f]_{E^\beta_a}]_{\Ee_{\beta}})=[a,[f]_{E^{\beta'}_a}]_{\Ee_{\beta'}}$$
is well-defined and elementary. Moreover, it commutes with the embeddings  $j_{\Ee_\beta}:V\to M_{\Ee_\beta}$ and $j_{\Ee_{\beta'}}:V\to M_{\Ee_{\beta'}}$. Let $M$ be the direct limit of 
$$\langle \langle M_{\Ee_\beta}:\beta \in I' \rangle, \langle k_{\beta, \beta '}: \beta <\beta' \mbox{ in }I_S\rangle\rangle$$   
and let $j_M :V\to M$ be the corresponding direct limit elementary embedding, which is given by 
$$j_M(x)=[\beta , [a,[c^a_x]_{E^\beta_a}]_{\Ee_\beta}]$$
for some (any) $a\in [\beta]^{<\omega}$. Let $\pi^M:M\to N$ be the transitive collapse, and let $j_N=\pi^M\circ j_M:V\to N$.

\medskip

Let $\xi=\mbox{sup}(I_S)$. Note that, as $\gamma \in I_S$,  $\xi >\kappa$.

\begin{claim}
\label{claimcrit}
$j_N(\kappa)=\xi$
\end{claim}

\begin{proof}[Proof of claim]
As in \cite{BW}, we can show that $j_{N_\beta}(\kappa)\geq \beta$, for every $\beta \in I_S\setminus \kappa$. So, for such a  $\beta$, letting $\ell_{\beta ,N}$ be the unique elementary embedding such that   $j_N=\ell_{\beta,N}\circ j_{N_\beta}$, we have:
$$j_N(\kappa)=\ell_{\beta,N}(j_{N_\beta}(\kappa))\geq \ell_{\beta , N}(\beta)\geq \beta.$$
Hence, $j_N(\kappa)\geq \xi$. 
Also,    $j_{E^\beta_a}(\kappa)$ can be computed in $V_\beta$, for all $a\in [\beta]^{<\omega}$, and therefore  $j_{N_\beta}(\kappa)\leq \beta$. Hence, $j_N(\kappa)\leq \xi$. 
\end{proof}

Since $\kappa <\xi$, it follows from the claim above that $\crit(j_N)\leq \kappa$. But  since for $\beta <\kappa$ the map $j_{N_\beta}$  is the identity, we must have  $\crit(j_N)=\kappa$. Also, since  $\gamma \in I_S$, $V_\gamma \subseteq N_\gamma$, hence  $V_{\gamma}\subseteq N$. This shows that $\kappa$ is $\gamma$-strong, as wanted.
\end{proof}

From proposition \ref{prop1} and theorem \ref{mainstrong} we obtain now the following characterization of strong cardinals.

\begin{corollary}
A cardinal $\kappa$ is strong if and only if it witnesses $\Pi_1$-$\PSR$ (second version).
\end{corollary}

Similar results  can be proven for $\Gamma$-strong cardinals (definition \ref{defSigmaStrong}). On the one hand, proposition \ref{prop6} shows that if $\kappa$ is $\Pi_n$-strong, then $\kappa$ witnesses $\mathbf{\Pi_{n+1}}$-$\PSR$. On the other hand,  similarly as in \ref{mainstrong}, we can prove that if $\kappa$ witnesses $\Pi_n$-$\PSR$, then $\kappa$ is $\Pi_n$-strong.
This yields the following characterization of $\Pi_n$-strong cardinals:

\begin{theorem}
\label{thmnstrong}
For every $n>0$, a cardinal $\kappa$ is $\Pi_n$-strong if and only it it witnesses $\Pi_n$-$\PSR$ (second version).
\end{theorem}

\subsection{Strong Product Structural reflection}

Let us consider next the following, arguably more natural, strengthening of $\PSR$:

\begin{quotation}
\begin{itemize}
\item[$\SPSR$:] (\emph{Strong Product Structural Reflection}) For every definable class of relational structures $\mathcal{C}$ of the same type, $\tau$, there exists an ordinal $\alpha$ that   {\emph{strongly product-reflects}} $\mathcal{C}$, i.e.,  
for every $\Ae$ in $\mathcal{C}$ there exists an ordinal  $\beta$ with $\Ae \in V_\beta$ and an elementary embedding $j:\prod (\Ce \cap V_\alpha)\to \prod (\Ce \cap V_\beta)$.
\end{itemize}
\end{quotation}

Similarly as in the case of $\PSR$, let us say that a cardinal $\kappa$ witnesses $\Gamma (P)$-$\SPSR$  if $\kappa$ strongly product-reflects all classes $\Ce$ that are $\Gamma$-definable (with parameters in $P$).
Also, a cardinal $\kappa$ witnesses $\mathbf{\Pi_n}$-$\SPSR$ if and only if it witnesses $\mathbf{\Sigma_{n+1}}$-$\SPSR$, and similarly for the lightface definability classes. Moreover, we obtain equivalent principles by restricting to classes of natural structures. Thus, we may formally define   $\Gamma$-$\SPSR$, for $\Gamma$ a lightface definability class, as follows:
\begin{quotation}
\begin{itemize}
\item[$\Gamma$-$\SPSR$:] (\emph{$\Gamma$-Strong Product Structural Reflection})    There exists a   cardinal $\kappa$ that  strongly product-reflects all $\Gamma$-definable class $\Ce$ of natural structures. 
\end{itemize}
\end{quotation}
The boldface version being as follows:
\begin{quotation}
\begin{itemize}
\item[]    There exist a proper class of cardinals $\kappa$ that  strongly product-reflect all $\Gamma$-definable, with parameters in $V_\kappa$, class $\Ce$ of natural structures. 
\end{itemize}
\end{quotation}

Note that $\Gamma$-$\SPSR$ implies $\Gamma$-$\PSR$, for any definability class $\Gamma$.

\medskip

We shall see next that the large cardinal notions that correspond to the $\SPSR$ principle are those of superstrong, globally superstrong, and $C^{(n)}$-globally superstrong cardinals.

\begin{definition}\cite{CT} 
A cardinal $\kappa$ is \emph{superstrong above $\lambda$}, for some $\lambda \geq \kappa$, if there exists an elementary embedding $j:V\to M$, with $M$ transitive, $\crit(j)=\kappa$, $j(\kappa)>\lambda$, and $V_{j(\kappa)}\subseteq M$. 

A cardinal $\kappa$ is \emph{globally superstrong} if it is superstrong above $\lambda$, for every $\lambda \geq \kappa$.

More generally, 
a cardinal $\kappa$ is \emph{$C^{(n)}$-superstrong above $\lambda$}, for some $\lambda \geq \kappa$, if there exists an elementary embedding $j:V\to M$, with $M$ transitive, $\crit(j)=\kappa$, $j(\kappa)>\lambda$, $V_{j(\kappa)}\subseteq M$, and $j(\kappa)\in C^{(n)}$. 

A cardinal $\kappa$ is \emph{$C^{(n)}$-globally superstrong} if it is $C^{(n)}$-superstrong above $\lambda$, for every $\lambda \geq \kappa$.
\end{definition}  

Note that every globally superstrong cardinal is $C^{(1)}$-globally superstrong. Also, every globally superstrong cardinal is superstrong, and every $C^{(n)}$-globally superstrong cardinal belongs to $C^{(n+2)}$ (\cite{CT}).
 As shown in \cite{CT}, on the one hand, if $\kappa$ is $C^{(n)}$-gobally superstrong, then there are many $C^{(n)}$-superstrong cardinals below $\kappa$. On the other hand, if $\kappa$ is $\kappa+1$-extendible, then $V_\kappa$ satisfies that there is a proper class of $C^{(n)}$-globally superstrong cardinals, for every $n$. Moreover, if $\kappa$ is $C^{(n)}$-extendible, then there are many $C^{(n)}$-globally superstrong cardinals below $\kappa$.

Similarly as in Proposition \ref{prop1} we can prove the following:

\begin{proposition}
\label{prop2}
If $\kappa$ is $C^{(n)}$-globally superstrong, then it witnesses $\mathbf{\Pi_n}$-$\SPSR$.
\end{proposition}

\begin{proof}
Let $\kappa$ be $C^{(n)}$-globally superstrong and let $\Ce$ be a class of relational structures of the same type that is  definable by a $\Pi_n$ formula $\varphi (x)$, with parameters in $V_\kappa$.   

Given any $\mathcal{A}\in \Ce$, let $\lambda\in C^{(n)}$ be greater than or equal to $\kappa$ and with  $\mathcal{A}\in V_\lambda$.

Let $j:V\to M$ be an elementary embedding witnessing that $\kappa$ is $C^{(n)}$-superstrong above $\lambda$. Then  the restriction of $j$ to $\Ce \cap V_\kappa$ yields an elementary embedding
$$h:\prod (\Ce \cap V_\kappa)\to \prod (\{ X: M\models \varphi(X)\} \cap V^M_{j(\kappa)}).$$
Let $S:= \{ X: M\models \varphi(X)\} \cap V^M_{j(\kappa)}$. Since $j(\kappa)>\lambda$ and $j(\kappa)\in C^{(n)}$, we have $V_{j(\kappa)}\models \varphi(\Ae)$. Hence, since $\kappa \in C^{(n)}$ (in fact $\kappa \in C^{(n+2)}$ (\cite{CT}), by elementarity $V_{j(\kappa)}=V_{j(\kappa)}^M\preceq_{\Sigma_n} M$, and thus $M\models \varphi(\Ae)$.  It follows  that $\Ae \in S$. Moreover, if $\Be\in S$, then  $M\models \varphi(\Be)$. Hence, since  $V_{j(\kappa)}^M \preceq_{\Sigma_n} M$, we have that $V_{j(\kappa)}^M=V_{j(\kappa)}\models \varphi(\Be)$. Since $j(\kappa)\in C^{(n)}$, $\varphi(\Be)$ holds in $V$, and therefore $\Be \in \Ce$. This shows that $S=\Ce \cap V_{j(\kappa)}$, and so 
$$h:\prod (\Ce \cap V_\kappa)\to \prod (\Ce \cap V_{j(\kappa)}).$$
Hence, $h$ 
witnesses $\SPSR$ for $\Ae$.
\end{proof}

Observe that 
 since the function $h$ in the proof of the last proposition is the restriction of $j$ to $\prod (\Ce \cap V_\kappa)$, and $j$ is elementary, it preserves all first-order properties. In particular, it is $\subseteq$-chain-preserving; and since $\kappa=\crit (j)$, $h$ is faithful. Thus, taking into consideration our remarks from section \ref{SR}, as well as those made in the previous section before we stated  the second version of the $\PSR$ schema, we may reformulate $\Gamma$-$\SPSR$ for $\Gamma$ a lightface definability class as follows:

\begin{quotation}
\begin{itemize}
\item[$\Gamma$-$\SPSR$:] (\emph{$\Gamma$-Strong Product Structural Reflection. Second version}) There exists a  cardinal $\kappa$ that \emph{strongly-product-reflects} all $\Gamma$-definable classes $\Ce$ of natural structures, i.e.,  for every $\Ae\in \Ce$ there exists an ordinal $\beta$ with $\Ae \in V_\beta$  and  a faithful and  $\subseteq$-chain-preserving elementary embedding  $h:\prod (\Ce \cap V_\kappa)\to \prod (\Ce \cap V_\beta)$.
\end{itemize}
\end{quotation}
The corresponding version for boldface $\Gamma$ being:
\begin{quotation}
\begin{itemize}
\item[] There exist a proper class of cardinals $\kappa$ that \emph{strongly-product-reflect} all $\Gamma$-definable, with parameters in $V_\kappa$, proper classes $\Ce$ of natural structures.
\end{itemize}
\end{quotation}

\medskip

Arguing similarly as in the proof of Theorem \ref{mainstrong} (and \cite{BW}), we can  now prove the following: 

\begin{theorem}
\label{mainsuperstrong2}
For every $n>0$, there is a $\Pi_n$-definable, without parameters, class $\Ce$ of natural structures such that if a cardinal  $\kappa$ strongly-product-reflects $\Ce$, then $\kappa$  is a $C^{(n)}$-globally superstrong cardinal.
\end{theorem}

\begin{proof}
Let $\mathcal{C}$ be the $\Pi_n$-definable class of all structures  
$$\Ae_\alpha:=\langle V_{\lambda_\alpha}, \in ,  \alpha \rangle$$   where  $\alpha$ has uncountable cofinality and is the $\alpha$-th element of $C^{(n)}$,  and  
$\lambda_\alpha$ is the least  cardinal in $C^{(n)}$ greater than $\alpha$.

Let $\kappa$ witness $\SPSR$ for $\Ce$. 
Let $I:=\{ \alpha :\Ae_\alpha \in V_\kappa \}$. 
Since $\kappa \in C^{(n+1)}$ (proposition \ref{Sigma2}), ${\rm{sup}}(I)=\kappa$.

Pick any ordinal $\lambda \geq \kappa$ and let us show that $\kappa$ is $\lambda$-$C^{(n)}$-superstrong. 
Let $\Ae_\beta$ in $\Ce$ with $\lambda <\beta$.
Let  $\kappa'$ be such that $\Ae_\beta \in V_{\kappa'}$ and there is a faithful $\subseteq$-chain-preserving  elementary embedding
$$j:\prod (\Ce \cap V_\kappa)\to \prod(\Ce \cap V_{\kappa'}).$$
Let  $$h_\beta :\prod (\Ce \cap V_{\kappa'})\to \Ae_\beta $$ be the projection map and define
$k_\beta:V_{\kappa +1}\to V_{\beta+1}$
by:
$$k_\beta(X)=h_\beta (  j(\{ X\cap V_{\alpha}\}_{\alpha \in I})).$$ 
As in \ref{mainstrong}, for each $a\in [\beta]^{<\omega}$,   define $E^\beta_a$ by
$$X\in E^\beta_a \quad \mbox{ iff }\quad X\subseteq [\kappa]^{|a|}  \mbox{ and } a\in k_\beta(X)\, .$$
Then $E^\beta_a$ is an $\omega_1$-complete proper ultrafilter over $[\kappa]^{|a|}$, and so the ultrapower ${\rm{Ult}}(V,E^\beta_a)$  is well-founded. Furthermore, since $j$ is faithful, if $\beta \in I$, then $E^\beta_a$ is the principal ultrafilter generated by $\{ a\}$. Let 
$$j^\beta_a:V\to M^\beta_a \cong {\rm{Ult}}(V, E^\beta_a)$$
 and let  $\mathcal{E}_\beta:=\{ E^\beta_a: a\in [\beta]^{<\omega}\}$.  As in \cite{BW},  $\mathcal{E}_\beta$ is normal and coherent.
Let   $M_{\Ee_\beta}$ be  the  direct limit of $$\langle \langle M^\beta_a:a\in [\beta]^{<\omega}\rangle, \langle i^\beta_{ab}:a\subseteq b\rangle\rangle$$
where the $i_{ab}^\beta$ are the standard projection maps,    
and let $j_{\Ee_\beta} :V\to M_{\Ee_\beta}$ be the corresponding limit elementary embedding. 
As in \cite{BW}, 
$M_{\Ee_\beta}$ is  well-founded. 
So, let $\pi_\beta:M_{\Ee_\beta} \to N_\beta$ be the transitive collapse, and let $j_{N_\beta}=\pi \circ j_{\Ee_\beta}:V\to N_\beta$. We have that  $V_\beta \subseteq N_\beta$ and 
$j_{N_\beta}(\kappa)\geq \beta$ (see \cite{BW}).

By claim \ref{claimcoherent}, if $\beta \leq \beta'$ are in $I':=\{\alpha : \Ae_\alpha \in V_{\kappa'}\}$, then $E^\beta_a = E^{\beta'}_a$, for every $a\in [\beta]^{<\omega}$. Hence, for every $\beta <\beta'$ in $I'$, the map
$$k_{\beta ,\beta'}:M_{\Ee_{\beta}}\to M_{\Ee_{\beta'}}$$
given by
$$k_{\beta ,\beta'}([a, [f]_{E_a}]_{\Ee_{\beta}})=[a,[f]_{E_a}]_{\Ee_{\beta'}}$$
is well-defined, elementary, and  commutes with the embeddings  $j_{\Ee_\beta}:V\to M_{\Ee_\beta}$ and $j_{\Ee_{\beta'}}:V\to M_{\Ee_{\beta'}}$. Let $M$ be the direct limit of 
$$\langle \langle M_{\Ee_\beta}:\beta \in I' \rangle, \langle k_{\beta, \beta '}: \beta <\beta' \mbox{ in }I'\rangle\rangle$$   
and let $j_M :V\to M$ be the corresponding limit elementary embedding. Let $\pi^M:M\to N$ be the transitive collapse, and let $j_N=\pi^M\circ j_M:V\to N$.

Let $\xi=\mbox{sup}(I')$. Note that $\xi \in C^{(n)}$ and $\xi>\kappa$. 
As in claim \ref{claimcrit}, $j_N(\kappa)=\xi$, hence  $\crit(j_N)\leq \kappa$. But  since for $\beta \in I$ the map $j_{N_\beta}$  is the identity,  $\crit(j_N)=\kappa$. Also, since $V_{\xi}=\bigcup_{\beta \in I}V_\beta$, and $V_\beta \subseteq N_\beta$ for  all $\beta \in I$, it follows that  $V_{\xi}\subseteq N$. This shows that $\kappa$ is $\xi$-superstrong, hence also $\lambda$-$C^{(n)}$-superstrong, as wanted.
\end{proof}

We have thus proved the following:

\begin{theorem}
\label{mainsuperstrong3}
For every $n\geq 1$, the following are equivalent for any cardinal $\kappa$:
\begin{enumerate}
\item $\kappa$ witnesses $\Pi_n$-$\SPSR$
\item $\kappa$ is a $C^{(n)}$-globally superstrong cardinal.
\end{enumerate}
\end{theorem}

\begin{corollary}
The following are equivalent:
\begin{enumerate}
\item $\SPSR$, i.e., $\mathbf{\Pi_n}$-$\SPSR$ for every $n$.
\item $\Pi_n$-$\SPSR$ for every $n$.
\item There exists a $C^{(n)}$-globally superstrong cardinal, for every $n$.
\item There exist a proper class of $C^{(n)}$-globally superstrong cardinals.
\end{enumerate}
\end{corollary}

\medskip

\subsection{Bounded Product Structural Reflection}
Let us consider next some \emph{bounded} forms of $\PSR$. Namely, for $\Gamma$ a lightface definability class and any ordinal $\beta$ let:
\begin{quotation}
\begin{itemize}
\item[$\Gamma$-$\PSR_\beta$:]  There exists a cardinal $\kappa$ that $\beta$-\emph{product-reflects}  every $\Gamma$-definable proper class $\Ce$ of natural structures, i.e., for every $\Ae$ in $\Ce$ of rank $\leq \kappa +\beta$ there exists  a  set $S$ with $\Ae$ in $S$  and  an elementary embedding  $h:\prod (\Ce \cap V_\kappa)\to \prod S$.
\end{itemize}
\end{quotation}
Thus $\Gamma$-$\PSR$ holds if and only if there exists a cardinal $\kappa$ that witnesses $\Gamma$-$\PSR_\beta$ for all (equivalently, a proper class of) ordinals $\beta$. 

\smallskip

The following theorem shows that measurable cardinals can be characterized  in terms of  bounded $\PSR$.
\begin{theorem}
\label{PSRmeasurable}
The following are equivalent:
\begin{enumerate}
\item $\Pi_1$-$\PSR_1$ 
\item There exists a measurable cardinal.
\end{enumerate}

\end{theorem}

\begin{proof}
(1) implies (2): Let $\Ce$ be the $\Pi_1$-definable class of $\langle V_{\gamma }, \in \rangle$, $\gamma \geq \omega$. Let $\kappa$ witness $\PSR_1$ for $\Ce$, and let $S$ be a set that contains $V_{\kappa +1}$ such that there exists an elementary embedding
$$h:\prod (\Ce \cap V_{\kappa})\to \prod S.$$
Define
$k:V_{\kappa +1}\to V_{\kappa +2}$
by:
$$k(X)=h_{\kappa +1}(h(\{ X\cap V_\gamma\}_{\gamma <\kappa}))$$
where $h_{\kappa +1}$ is the projection on $V_{\kappa +1}$.
Then $k$ preserves Boolean operations and  the subset relation, and   is the identity on $\omega +1$. Moreover, $k(\kappa)=\kappa +1$.
Now for each $a\in V_{\kappa +1}$, define $\Ue_a$ by
$$X\in\Ue_a  \quad \mbox{ iff }\quad X\subseteq \kappa \mbox{ and } a \in k(X)\, .$$
Clearly, $\kappa \in \Ue_a$. Also, since $k$ preserves Boolean operations and the $\subseteq$ relation, $\Ue_a$ is a proper ultrafilter over $\kappa$. Moreover, since $k(\omega)=\omega$,  $\Ue_a$ is $\omega_1$-complete. Furthermore, since $|V_{\kappa +1}|=2^{|V_\kappa | } > |\kappa |$, some $\Ue_a$ is non-principal. So, some cardinal less than or equal to $\kappa$ is measurable.  

(2) implies (1): Let $\kappa$ be a measurable cardinal, and  let $\varphi (x)$ be a $\Pi_1$ formula (we may allow parameters in $V_\kappa$)  that defines a proper class $\Ce$ of natural structures. We claim that $\kappa$ $1$-product-reflects $\Ce$.

Let $j:V\to M$ be an ultrapower elementary embedding, given by some $\kappa$-complete normal measure over $\kappa$. Thus, $\crit (j)=\kappa$  and $V_{\kappa + 1}^M=V_{\kappa +1}$.
By elementarity, the restriction of $j$ to $\Ce \cap V_\kappa$ yields an elementary embedding
$$h:\prod (\Ce \cap V_\kappa)\to \prod (\{ X: M\models \varphi(X)\} \cap V^M_{j(\kappa)}).$$
Let $S:=  \{ X: M\models \varphi(X)\} \cap V^M_{j(\kappa)}$. If $\mathcal{A}=\langle V_\gamma ,\in\rangle \in \Ce$, with $\gamma \leq \kappa +1$, then $\Ae \in M$, and  since $\varphi (\Ae)$ holds in $V$, by $\Pi_1$ downward absoluteness for transitive classes it also holds in $M$. Moreover, since $j(\kappa)>\kappa +1$, $\Ae \in  V^M_{j(\kappa)}$, hence  $\Ae\in S$.
\end{proof}

Let us note that the proof of the theorem above also shows that the least measurable cardinal is precisely the least cardinal $\kappa$ that witnesses $\PSR_1$ for all classes $\Ce$ that are $\Pi_1$-definable with parameters in $V_\kappa$.

\section{Large cardinals below measurability}

We shall next consider  Structural Reflection for classes of relational structures that are $\Sigma_1$-definable in the language of set theory extended with additional $\Pi_1$ predicates. That is,  classes of structures of complexity between $\Sigma_1$ and $\Sigma_2$.

Let $\Rcal$ be a  set of $\Pi_1$ predicates or relations. A class $\Ce$ of structures in a fixed countable relational type is said to be $\Sigma_1(\Rcal)$-definable  if it is definable by means of a $\Sigma_1$ formula of the first-order language of set theory with additional predicate symbols for the predicates in  $\Rcal$, without parameters. We  define the following form of $\SR$:

\begin{quote}
\begin{itemize}
\item[$\Sigma_1(\Rcal)$-$\SR$:] 
For every $\Sigma_1(\Rcal)$-definable   class $\Ce$ of   structures  of the same type there exists a cardinal $\kappa$ that \emph{reflects} $\mathcal{C}$, i.e.,  for every $A$ in $\mathcal{C}$ there exists $B$ in $\mathcal{C}\cap V_\kappa$ and an elementary embedding from $B$ into $A$.\end{itemize}
\end{quote}


For the rest of this section we shall  write $\SR_\Rcal$ instead of the more cumbersome $\Sigma_1(\Rcal)$-$\SR$. Also, if $\Rcal=\{ R_1,\ldots ,R_n\}$, then we may write $\SR_{R_1,\ldots,R_n}$ for $\SR_\Rcal$.

We have that $\SR_\emptyset$, i.e., $\Sigma_1$-$\SR$, is provable in ZFC (Proposition \ref{sigma1}). However, if $R$ is the $\Pi_1$ relation ``$x$ is an ordinal and $y=V_x$", then $\SR_R$ holds if and only if there exists a supercompact cardinal (\cite{Ba:CC, BCMR}; see also Theorem \ref{thmpi1}). Moreover, if $\kappa$ is supercompact, then $\SR_\Rcal$ holds for  $\kappa$, for any set $\Rcal$ of $\Pi_1$ predicates (cf. Theorem \ref{thmpi1}).

\subsection{The principle $\SR^-_\Rcal$}
\label{subsection8.1}

For $\Gamma$ any lightface definability class, the following is a natural  restricted form of $\Gamma$-$\SR$:

\begin{quote}
\begin{itemize}
\item[$\Gamma$-$\SR^-:$] 
There exists a cardinal $\kappa$ such that for every $\Gamma$-definable   class $\Ce$   of structures of the same type and every $A\in \Ce$ of cardinality $\kappa$ there exists $B\in \Ce\cap H_\kappa$ and an elementary embedding from $B$ into $A$. We say that the cardinal $\kappa$ $\kappa$-\emph{reflects} $\mathcal{C}$.
\end{itemize}
\end{quote}

The restriction of $\SR^-$ to $\Sigma_1(\mathcal{R})$-definable classes of structures was first introduced in \cite{BV}. Namely, for $\Rcal$ a finite set of $\Pi_1$ predicates or relations, let 

\begin{quote}
\begin{itemize}
\item[$\SR^-_\Rcal:$] 
There exists a cardinal $\kappa$ that $\kappa$-reflects every  $\Sigma_1(\Rcal)$-definable  (with  parameters in $H_\kappa$) class $\Ce$ of structures of the same type.
\end{itemize}
\end{quote}


\subsubsection{The Cardinality predicate}
Let $\Cd$ be the $\Pi_1$ predicate $``x$ is a cardinal". 
Magidor and V\"a\"an\"anen  \cite{MaVa} show that  the principle $\SR_{\Cd}$ implies $0^\sharp$, and much more, e.g., there are no good scales.
The principle $\SR^-_{\Cd}$ is much weaker, but it does have some large-cardinal strength, as the next theorem shows.

\begin{theorem}[\cite{BV}]
\label{weakCd}
$\SR^-_{\Cd}$ holds, witnessed by $\kappa$, then   there exists a weakly inaccessible cardinal $\lambda \leq \kappa$.
\end{theorem}

It is shown in \cite{MaVa} that, starting form a supercompact cardinal, one can produce a model of ZFC in which $\SR_{\Cd}$ holds for the first weakly inaccessible cardinal. Thus,  no large-cardinal properties beyond  weak inaccessibility   may be proved in ZFC to hold for the least cardinal witnessing $\SR_{\Cd}$.

\subsubsection{The Regularity predicate}

Let $\Rg$ be the $\Pi_1$ predicate $``x$ is a regular ordinal".

\begin{theorem}[\cite{BV}]
\label{propRg}
If $\SR^-_{\Rg}$ holds, witnessed by $\kappa$, then there exists a weakly Mahlo cardinal $\lambda \leq \kappa$.
\end{theorem}

It follows from \cite{MaVa}  that one cannot hope to get from $\SR_{\Rg}$ more than a  weakly Mahlo cardinal $\leq \kappa$, for starting from a weakly Mahlo cardinal one can obtain a model in which $\SR_{\Rg}$ is witnessed by the least weakly Mahlo cardinal. One cannot hope either to show that the least $\kappa$ witnessing $\SR_{\Rg}$  is strongly inaccessible, for in \cite{SV} it is shown that one can have $\SR_{\Rg}$ witnessed by $\kappa =2^{\aleph_0}$.

Let us note that, since the predicate $\Cd$ is  $\Sigma_1(\Rg)$-definable (see \cite{BV}), the principle  $\SR_{\Cd ,\Rg}$ is equivalent to $\SR_{\Rg}$. 

\subsubsection{The Weakly Inaccessible predicate}

There is a principle between $\SR^-_{\Cd}$ and $\SR^-_{\Rg}$, namely $\SR^-_{\Cd, \WI}$, where 
$\WI$ is the $\Pi_1$ predicate $``x$ is weakly inaccessible".

\begin{proposition}[\cite{BV}]
\label{propWIg}
If $\SR^-_{\Cd, \WI}$ holds, witnessed by $\kappa$, then    there exists a  $2$-weakly  inaccessible cardinal $\lambda \leq \kappa$.
\end{proposition}

We may also consider predicates $\alpha$-$\WI$, for $\alpha$ an ordinal. That is, the predicate ``$x$ is $\alpha$-weakly inaccessible". Then, similar arguments as in \cite{BV} would show that the principle $\SR^-_{\Cd, \, \alpha \,\mbox{-} \WI}$ holding for $\kappa$ implies that there is an $(\alpha +1)$-weakly inaccessible cardinal $\lambda\leq \kappa$.  

 \subsubsection{Weak compactness}
 
 Let $\WC(x,\alpha)$ be the $\Pi_1$ relation $``\alpha$ is a limit ordinal and $x$ is a partial ordering with no chain of  order-type $\alpha$".

\begin{theorem}[\cite{BV}]
\label{propWC}
If $\SR^-_{\Cd, \WC}$ holds, witnessed  by some $\kappa$ such that if $\gamma \leq \kappa$ is weakly inaccessible, then $2^\delta \leq \gamma$ for all cardinals $\delta <\gamma$, then there exists a weakly compact cardinal $\lambda \leq \kappa$.\footnote{In \cite{BV}  it is only assumed that $\kappa$ witnesses $\SR^-_{\Cd, \WC}$. However, Philipp L\"ucke \cite{Lu:SR} has shown that some additional assumption on $\kappa$ is needed. Note that our assumption on $\kappa$ in the current statement of the Theorem follows form the GCH.} 
\end{theorem}

Since the first weakly Mahlo cardinal may satisfy $\SR_{\Rg}$ (\cite{MaVa}), we cannot prove the existence of a weakly compact cardinal $\leq \kappa$ just from $\SR_{\Rg}$. Hence, $\SR_{\Cd, \WC}$ is stronger than $\SR_{\Rg}$.

\medskip

Let $\PW$ be the $\Pi_1$ relation $\{ (x,y): y=\mathcal{P}(x)\}$.  Then we have the following:

\begin{theorem}[\cite{Mag}]
$\kappa$ is the least cardinal witnessing $\SR_{\PW}$ if and only if $\kappa$ is the first supercompact cardinal.
\end{theorem}

It follows from Theorem \ref{thmpi1} that $\SR_{\PW}$ is in fact equivalent to $\Pi_1$-$\SR$, and also equivalent to $\Sigma_2$-$\SR$.
However,  L\"ucke \cite{Lu:SR} has established  that $\SR^-_{\PW}$ is much weaker than $\SR_{\PW}$. Indeed, he shows that  $\SR^-_{\PW}$ is  equivalent to the existence of a weakly shrewd cardinal, a large cardinal notion obtained by weakening the definition of shrewd cardinal studied  by M. Rathjen in \cite{Ra:RAOA} and whose consistency strength is strictly between the large cardinal notions of total indescribability and subtleness. However, as shown in \cite{Lu:SR},  shrewd and weakly shrewd cardinal are equiconsistent.

\begin{definition}[\cite{Lu:SR}] A cardinal $\kappa$ is \emph{weakly shrewd} if for every formula $\varphi(x,y)$ of the language of set theory, every cardinal $\theta >\kappa$, and every $A\subseteq \kappa$ such that $\varphi(A,\kappa)$ holds in $H_\theta$, there exist cardinals $\bar{\kappa} <\bar{\theta}$ such that $\bar{\kappa}<\kappa$ and $\varphi (A\cap \bar{\kappa}, \bar{\kappa})$ holds in $H_{\bar{\theta}}$.
\end{definition}

\begin{theorem}[\cite{Lu:SR}] The following are equivalent:
\begin{enumerate}
\item $\kappa$ is the least weakly shrewd cardinal.
\item $\kappa$ is the least cardinal witnessing $\SR^-_{\PW}$.
\item $\kappa$ is the least cardinal witnessing  $\Sigma_2$-$\SR^-$.
\end{enumerate}
\end{theorem}

Since, as shown in \cite{Lu:SR}, weakly shrewd cardinals may be smaller than $2^{\aleph_0}$,   the principle $\Sigma_2$-$\SR^-$, and therefore also $\SR^-_{\mathcal{R}}$, for any set $\mathcal{R}$ of $\Pi_1$ predicates, does not imply the existence of a strongly inaccessible cardinal. Moreover, \cite{Lu:SR} shows that it is consistent, modulo the existence of a weakly shrewd cardinal that is not shrewd (a large cardinal notion consistency-wise weaker than subtleness), that there exists a cardinal less than $2^{\aleph_0}$ witnessing the principle $\SR^-$ for all definable classes of structures of the same type, taken as a schema, i.e., $\Sigma_n$-$\SR^-$, for all $n<\omega$. Thus, even $\SR^-$ cannot  imply the existence of a strongly inaccessible cardinal.

\subsection{Strong $\Sigma_1(\mathcal{R})$-definability}

Notice that a  class $\Ce$ is $\Sigma_1$-definable  iff there is a $\Sigma_1$ formula $\varphi$ such that for every $A$, $A\in \Ce$ if and only if some transitive structure $\langle M,\in\rangle$ that contains $A$ satisfies $\varphi(A)$. Now let $\mathcal{L}_{\dot{R}}$ be the language of set theory expanded with an additional predicate symbol $\dot{R}$, and suppose $R$ is a predicate. Naturally, one may define a class $\Ce$ to be $\Sigma_1(R)$ if it is $\Sigma_1$-definable in the language $\mathcal{L}_{\dot{R}}$ with $\dot{R}$ being interpreted as $R$. However, unlike the case of $\Sigma_1$-definability, this is not equivalent to saying that there is a $\Sigma_1$ formula $\varphi$ of the language $\mathcal{L}_{\dot{R}}$ such that for every $A$, $A\in \Ce$ if and only if some transitive structure $\langle M,\in, R^M \rangle$ that contains $A$ satisfies $\varphi(A)$. For the equivalence to hold we need to require that $R^M$ is precisely $R\cap M$. Namely,

\begin{proposition}
The following are equivalent for all classes   $\Ce$ and predicates $R$:
\begin{enumerate}
\item $\Ce$ is $\Sigma_1(R)$, i.e., there exists a $\Sigma_1$ formula $\varphi(x)$ of $\mathcal{L}_{\dot{R}}$ such that $$\Ce=\{ A:\varphi(A), \mbox{with $\dot{R}$  interpreted as $R$}\}$$
\item There is a $\Sigma_1$ formula $\varphi (x)$ of the language $\mathcal{L}_{\dot{R}}$ such that for every $A$, $A\in \Ce$ if and only if $$\langle M,\in, R\cap M \rangle \models \varphi(A)$$
for some transitive structure  $\langle M,\in \rangle$ that contains $A$.
\end{enumerate}

\end{proposition}

Notice also that if $\Ce$ is a $\Sigma_1$-definable class of structures of the same type, then the closure of $\Ce$ under isomorphisms is also $\Sigma_1$-definable, and we have the following equivalences:

\begin{proposition}
The following are equivalent for any   class  $\Ce$ of structures of the same type that is closed under isomorphisms:
\begin{enumerate}
\item $\Ce$ is $\Sigma_1$.
\item There is a $\Sigma_1$ formula $\varphi (x)$ of the language of set theory such that for every $A$, $A\in \Ce$ if and only if $$\langle M,\in  \rangle \models \varphi(B)$$
for some  transitive structure  $\langle M,\in  \rangle$ of size $|A|$ that contains $B$, where $B$ is isomorphic to $A$.
\end{enumerate}

\end{proposition}

Based on the considerations above, the following is therefore a natural definition for a class of structures closed under isomorphisms to be  $\Sigma_1$-definable with an additional predicate $R$. This is a  reformulation, for the case $n=1$, of L\"ucke's \cite{Lu:SR} definition of \emph{local $\Sigma_n(R)$-class}:

\begin{definition}
A class $\Ce$ of structures of the same type and closed under isomorphisms is $\Sigma_1(R)^\ast$ if 
there is a $\Sigma_1$   formula $\varphi (x)$ of the language $\mathcal{L}_{\dot{R}}$ such that for every $A$, $A\in \Ce$ if and only if $$\langle M,\in,   R \cap M \rangle \models \varphi(B)$$
for some  transitive structure  $\langle M,\in  \rangle$ of size $|A|$ that contains $B$, where $B$ is isomorphic to $A$.
\end{definition}

Observe  that  although  every $\Sigma_1(R)^\ast$ class $\Ce$ is $\Sigma_1(R)$, the converse is not true, even assuming closure under isomorphisms. An example is the class $\Ce$ of all structures isomorphic to some transitive $\langle M,\in, {\rm{Cd}}\cap M\rangle$, where ${\rm{Cd}}$ is the class of cardinals.
%

The closure under isomorphisms of the  $\Sigma_1(\mathcal{R})$-definable classes of structures  that are used in the proofs of \ref{weakCd}, \ref{propRg}, \ref{propWIg},
and  \ref{propWC} (as given in \cite{BV}) are easily seen to be $\Sigma_1(\mathcal{R})^\ast$, for the corresponding  $\mathcal{R}$. Thus, the results follow from the weaker $\Sigma_1(\mathcal{R})^\ast$-$\SR^-$ corresponding assumptions. Also, the argument in the proof of Theorem 5.5 from  \cite{BV} can be adapted to show that if $\mathcal{L}^\ast$ and $\mathcal{R}$ are symbiotic, then the  $SLST(\mathcal{L}^\ast)$ property  implies $\Sigma_1(\mathcal{R})^\ast$-$\SR^-$ (see \cite{BV}).\footnote{However, as shown in \cite{Lu:SR}, it does not imply $(\SR)^-_\mathcal{R}$,  as claimed in \cite{BV}.} Thus, from the results in section 8 of \cite{BV} one may obtain  the following equivalences:

\begin{theorem}
$ $
\begin{enumerate}
\item $($\cite{Lu:SR}$)$ $\kappa$ is the least weakly inaccessible cardinal iff $\kappa$ is the least cardinal witnessing $\Sigma_1(\Cd)^\ast$-$\SR^-$.
 
\item $($\cite{Lu:SR}$)$ $\kappa$ is the least weakly Mahlo cardinal iff $\kappa$ is the least cardinal witnessing $\Sigma_1(\Rg)^\ast$-$\SR^-$.

\item $\kappa$ is the least $\alpha$-weakly inaccessible cardinal iff $\kappa$ is the least cardinal witnessing $\Sigma_1(\Cd, \alpha$-$\WI)^\ast$-$\SR^-$.

\item Suppose $\kappa$ is such that if $\gamma \leq \kappa$ is weakly inaccessible, then $2^\delta \leq \gamma$ for all cardinals $\delta <\gamma$. Then $\kappa$ is the least weakly compact cardinal iff $\kappa$ is the least cardinal witnessing $\Sigma_1(\Cd, \WC)^\ast$-$\SR^-$.

\end{enumerate}
\end{theorem}

Under the assumption of GCH, or  just assuming that every weakly inaccessible cardinal is  inaccessible, the theorem above yields exact characterizations in terms of $\SR$ for the first inaccessible, Mahlo, $\alpha$-inaccessible, and weakly-compact cardinals.


\begin{corollary}[GCH]
$ $
\begin{enumerate}
\item  $\kappa$ is the least  inaccessible cardinal iff $\kappa$ is the least cardinal witnessing $\Sigma_1(Cd)^\ast$-$\SR^-$.
 
\item  $\kappa$ is the least  Mahlo cardinal iff $\kappa$ is the least cardinal witnessing $\Sigma_1(Rg)^\ast$-$\SR^-$.

\item $\kappa$ is the least $\alpha$-inaccessible cardinal iff $\kappa$ is the least cardinal witnessing $\Sigma_1(Cd, \alpha$-$WI)^\ast$-$\SR^-$.

\item  $\kappa$ is the least weakly compact cardinal iff $\kappa$ is the least cardinal witnessing $\Sigma_1(Cd, WC)^\ast$-$\SR^-$.

\end{enumerate}
\end{corollary}

In items (1)-(4) above one may, equivalently, strengthen $\Sigma_1(\mathcal{R})^\ast$-$\SR^-$ by allowing  $\kappa$-reflection for classes of structures $\Ce$ that are $\Sigma_1(\mathcal{R})^\ast$-definable with  parameters in $H_\kappa$.

\section{Generic Structural Reflection}

If $A$ and $B$ are structures of the same type, we say that an elementary embedding $j:A\to B$ is  \emph{generic} if it exists in some forcing extension of $V$. We shall next consider the following \emph{generic} version of $\SR$:

\begin{quote}
\begin{itemize}
\item[$\GSR$:] (\emph{Generic  Structural Reflection})
For every definable (with parameters) class $\mathcal{C}$ of relational structures  of the same type there exists an ordinal $\alpha$ that \emph{generically-reflects} $\mathcal{C}$, i.e.,  for every $A$ in $\mathcal{C}$ there exists $B$ in $\mathcal{C}\cap V_\alpha$ and a generic elementary embedding from $B$ into $A$.
\end{itemize}
\end{quote}
Thus, $\GSR$ is just like $\SR$, but the elementary embeddings may not exist in $V$ but in some  forcing extension of $V$. The next proposition shows that this is equivalent to requiring that the elementary embedding exists in any forcing extension resulting from collapsing the structure $B$ to make it countable.

\begin{prop}[(\cite{BGS}]
The following are equivalent for structures $B$ and $A$ of the same type.
\begin{enumerate}
\item $V^{\Coll(\omega,B)}\models``\text{There is an elementary embedding }j:B\to A."$
\item For some forcing notion $\Pe$,  \newline $V^{\Pe}\models``\text{There is an elementary embedding }j:B\to A."$
\end{enumerate}
\end{prop}

Taking into account similar considerations as in the case of $\SR$ and $\PSR$, we may properly formulate $\GSR$ as a schema. Namely, for $\Gamma$ a lightface definability class, let:

\begin{quote}
\begin{itemize}
\item[$\Gamma$-$\GSR$:] (\emph{$\Gamma$-Generic  Structural Reflection}) There exists a cardinal  $\kappa$ that \emph{generically-reflects} all $\Gamma$-definable, with parameters in $V_\kappa$, classes $\mathcal{C}$ of natural structures, i.e.,  
 for every $A$ in $\mathcal{C}$ there exists $B$ in $\mathcal{C}\cap V_\alpha$ such that in $V^{\Coll(\omega,B)}$ there is an elementary embedding from $B$ into $A$.
\end{itemize}
\end{quote}
The boldface version being:
\begin{quote}
\begin{itemize}
\item[] There exist a proper class of cardinals $\kappa$ that \emph{generically-reflect} all $\Gamma$-definable, with parameters in $V_\kappa$, classes $\mathcal{C}$ of natural structures.
\end{itemize}
\end{quote}

The assertion that $\kappa$ witnesses $\Gamma$-$\GSR$, for $\Gamma$ a boldface definability class, is equivalent to the \emph{Generic Vop\v{e}nka Principle} ${\rm{gVP}}(\kappa ,\Gamma)$  introduced in \cite{BGS}.

Similar considerations as in the case of $\SR$ (see the remarks before and after proposition \ref{equiv}) show that $\mathbf{\Pi_n}$-$\GSR$ and $\mathbf{\Sigma_{n+1}}$-$\GSR$ are equivalent; and also $\Pi_n$-$\GSR$ and $\Sigma_{n+1}$-$\GSR$ are equivalent.  

\medskip

We shall see next that some large cardinals, such as Schindler's \emph{remarkable} cardinals, can be characterized in terms of $\GSR$.

\begin{definition}[\cite{Sch:PFRC, Sch:RC}]
\label{defrem}
A cardinal $\kappa$ is \emph{remarkable}  if for every regular cardinal $\lambda>\kappa$, there is a regular cardinal $\bar{\lambda} <\kappa$ such that in $V^{\Coll(\omega,\lt\kappa)}$ there is an elementary embedding $j:H_{\bar{\lambda}}^V\to H_\lambda^V$ with $j(\crit(j))=\kappa$. 
\end{definition}

 A cardinal is remarkable if and only if it is $1$-remarkable (Definition \ref{n-rem}). Remarkable cardinals are downward absolute to $L$ and their consistency strength is strictly below a $2$-iterable cardinal.  
Remarkable cardinals are in $C^{(2)}$, and they are totally indescribable and ineffable, hence limits of totally indescribable cardinals (see \cite{BGS}). 

\begin{theorem}[\cite{BGS}]
\label{thmnrem}
The following are equiconsistent:
\begin{enumerate}
\item $\Pi_1$-$\GSR$
\item There exists a cardinal $\kappa$ that witnesses $\mathbf{\Pi_1}$-$\GSR$ 
\item There exists a remarkable cardinal.
\end{enumerate}
\end{theorem}

Let us say that a cardinal $\kappa$ is \emph{almost remarkable} if it is almost-$1$-remarkable (Definition {\ref{defarem}), namely:   for all $\lambda >\kappa$ in $C^{(1)}$ and every $a\in V_\lambda$, there is $\bar{\lambda} <\kappa$ also in $C^{(1)}$ such that in $V^{\rm{Coll}(\omega , <\kappa)}$ there exists an elementary embedding $j:V_{\bar{\lambda}}\to V_{\lambda}$ with  $a\in {\rm{range}}(j)$.
Then theorem \ref{remgsr} yields the following:

\begin{theorem}
A cardinal $\kappa$ witnesses $\mathbf{\Pi_1}$-$\GSR$ if and only if $\kappa$ is almost remarkable.
\end{theorem}

It follows that the notions of remarkable cardinal and of almost-remarkable cardinal are equiconsistent.

\medskip
Magidor \cite{Mag} shows that a cardinal $\kappa$ is supercompact if and only if for every regular cardinal $\lambda >\kappa$ there is a regular cardinal $\bar{\lambda} <\kappa$ and an elementary embedding $j:H_{\bar{\lambda}}\to H_\lambda$ with $j(\crit(j))=\kappa$. One can thus view a remarkable cardinal as a \emph{virtually supercompact}\footnote{We choose to call it \emph{virtually supercompact}, as in \cite{BGS}, instead of the perhaps more natural \emph{generic supercompact}, for the latter notion  already exists in the literature with a different meaning.}  cardinal. In analogy with theorem \ref{thmpi1} one might therefore expect $\Pi_1$-$\GSR$ to be not just equiconsistent, but actually  equivalent with the existence of a remarkable cardinal. Even more, since the first supercompact cardinal is precisely the first cardinal that witnesses $\mathbf{\Pi_1}$-$\SR$, one might conjecture that the first remarkable cardinal is the first cardinal that witnesses $\mathbf{\Pi_1}$-$\GSR$. This is almost true, but not exactly. On the one hand, if $\kappa$ is a remarkable cardinal, then $\kappa$ witnesses $\mathbf{\Pi_1}$-$\GSR$ (\cite{BGS}).  On the other hand, as shown in theorem \ref{firstremarkable} below, if there is no $\omega$-Erd\"os cardinal in $L$, then the least cardinal witnessing $\mathbf{\Pi_1}$-$\GSR$ is also the first remarkable cardinal.

The following equivalent definition of remarkability was given in \cite{BGS}: a cardinal $\kappa$ is remarkable if and only if for every $\lambda >\kappa$ there exists some $\bar{\lambda}<\kappa$ and a generic elementary embedding $j:V_{\bar{\lambda}}\to V_\lambda$ with $j(\crit(j))=\kappa$.

Wilson \cite{W} defines the notion of weakly remarkable cardinal by not requiring that $\bar{\lambda}$ is strictly below $\kappa$. Namely,

\begin{definition}[\cite{W}]
$\kappa$ is \emph{weakly remarkable} if and only if for every $\lambda >\kappa$ there exists some $\bar{\lambda}$ and a generic elementary embedding $j:V_{\bar{\lambda}}\to V_\lambda$ with $j(\crit(j))=\kappa$.
\end{definition}

Wilson shows  that if there exists a weakly remarkable non-remarkable, cardinal $\kappa$, then some ordinal greater than $\kappa$ is an $\omega$-Erd\"os cardinal in $L$. Moreover, the statements ``There exists an $\omega$-Erd\"os cardinal" and ``There exists a weakly remarkable non-remarkable cardinal" are equiconsistent modulo ZFC, and equivalent assuming $V=L$. 

Observe that  if $\kappa$ is cardinal  witnessing $\mathbf{\Pi_1}$-$\GSR$, then $\kappa$ also witnesses $\mathbf{\Pi_1}$-$\GSR$ in $L$. Thus, if $\kappa$ witnesses  $\mathbf{\Pi_1}$-$\GSR$ and in $L$ $\lambda$ is the least inaccessible  cardinal above $\kappa$, then $L_\lambda$  is a   model of ZFC in which $\kappa$ satisfies  $\Pi_1$-$\GSR$ and  there is no $\omega$-Erd\"os cardinal above $\kappa$. By combining arguments from \cite{BGS}  and \cite{W} we have the following:

\begin{theorem}
\label{firstremarkable}
Assume there is no $\omega$-Erd\"os cardinal in $L$. Then, the least cardinal that satisfies $\Pi_1$-$\GSR$, if it exists,  is remarkable.

\end{theorem}

\begin{proof}
Let $\kappa$ be the least cardinal witnessing $\Pi_1$-$\GSR$. Let $\Ce$ be the $\Pi_1$-definable class of structures of the form $\langle V_{\lambda +1},\in \rangle$ with $\lambda \in C^{(1)}$. Pick a singular cardinal $\lambda \in C^{(2)}$ greater than $\kappa$. By $\Pi_1$-$\GSR$, let $j:V_{\bar{\lambda}+1}\to V_{\lambda+1}$ be a generic elementary embedding with $\bar\lambda <\kappa$. Let $\bar\alpha =\crit(j)$. Note that $\bar\alpha <\bar\lambda$, because $\bar\alpha$ is regular and $\bar\lambda$ is not.

We claim that $\bar\alpha$ is weakly remarkable up to $\bar\lambda$. So, fix some $\delta >\bar\alpha$ smaller than $\bar\lambda$. Consider the restriction $j:V_\delta \to V_{j(\delta)}$, which has $j(\crit(j))=j(\bar\alpha)$. Then $V_{\lambda +1}$ satisfies that for some $\bar\delta$ there exists a generic elementary embedding $j^\ast :V_{\bar\delta}\to V_{j(\delta)}$ such that $j^\ast(\crit(j^\ast))=j(\bar\alpha)$. Hence, by elementarity $V_{\bar\lambda +1}$ satisfies that for some $\bar\delta$ there exists a generic elementary embedding $j^\ast :V_{\bar\delta}\to V_\delta$ with $j^\ast(\crit(j^\ast))=\bar\alpha$.

By elementarity, $\alpha:= j(\bar\alpha)$ is weakly remarkable up to $\lambda$; and since $\lambda \in C^{(2)}$, $\alpha$ is weakly remarkable. Since the existence of a weakly remarkable non-remarkable cardinal implies the existence of an $\omega$-Erd\"os cardinal in $L$ (\cite{W}), by our assumption we have that $\alpha$ is in fact remarkable. Hence, since every remarkable cardinal witnesses $\Pi_1$-$\GSR$ (\cite{BGS}), we have that $\kappa \leq \alpha$.

The theorem will be proved by showing that $\alpha =\kappa$. For suppose, aiming for a contradiction, that  $\kappa <\alpha$. Since $\alpha$ is remarkable and therefore belongs to $C^{(2)}$, we have 
$$V_\alpha \models ``\kappa \mbox{ witnesses }\Pi_1\mbox{-}\GSR".$$
By elementarity, there is some $\gamma <\bar\alpha$ such that
$$V_{\bar\alpha} \models ``\gamma \mbox{ witnesses }\Pi_1\mbox{-}\GSR".$$
Hence, since $j(\gamma)=\gamma$, again by elementarity,
$$V_\alpha \models ``\gamma \mbox{ witnesses }\Pi_1\mbox{-}\GSR"$$
and therefore $\gamma$ witnesses $\Pi_1$-$\GSR$, thus contradicting the minimality of $\kappa$.
\end{proof}

\begin{corollary}
Assume there is no $\omega$-Erd\"os cardinal in $L$. Then, the following are equivalent for a cardinal $\kappa$:
\begin{enumerate}
\item $\kappa$ is the least cardinal witnessing $\Pi_1$-$\GSR$.
\item $\kappa$ is the least cardinal witnessing $\mathbf{\Pi_1}$-$\GSR$.
\item $\kappa$ is the least almost remarkable cardinal.
\item $\kappa$ is the least weakly remarkable cardinal.
\item $\kappa$ is the least remarkable cardinal.
\end{enumerate}
\end{corollary}

We don't know if the assumption that there is no $\omega$-Erd\"os cardinal in $L$ is necessary for the equivalence above to hold. However, we have the following:

\begin{proposition}
For every $n>0$, if  $\kappa$ witnesses $\mathbf{\Pi_n}$-$\GSR$, then $\kappa \in C^{(n+1)}$. In particular, if $\kappa$ witnesses $\mathbf{\Pi_1}$-$\GSR$, then $\kappa \in C^{(2)}$.
\end{proposition}

\begin{proof}
Let us prove the case $n=1$. The general case follows by induction, using a similar argument. So, suppose $a\in V_\kappa$, $\varphi (x,y)$ is a $\Pi_1$ formula with $x,y$ as the only free variables, and $V\models \exists x \varphi (x,a)$. Pick $\lambda\in C^{(2)}$ greater than $\kappa$, so that $V_\lambda \models \exists x \varphi (x,a)$. Since the class of structures of the form $\langle V_\alpha ,\in ,a\rangle$ is $\Pi_1$-definable with $a$ as a parameter, there exists  a generic elementary embedding $j:\langle V_{\bar\lambda}, \in , a\rangle \to \langle V_\lambda ,\in ,a\rangle$   with $\bar\lambda <\kappa$. Note that, on the one hand,  since $\lambda$ belongs to $C^{(1)}$ so does $\bar\lambda$, hence by downward absoluteness for $\Pi_1$ sentences, $V_\kappa \models ``\bar\lambda \in C^{(1)}"$. On the other hand, by elementarity of $j$, $V_{\bar\lambda}\models \exists x \varphi (x,a)$. Hence by upwards absoluteness, $V_\kappa \models \exists x \varphi (x,a)$. 

A simpler similar argument, using the fact that $\Sigma_1$ sentences are absolute for transitive sets, shows that $\kappa \in C^{(1)}$. Hence, if $\varphi(x,y)$ and $a$ are as above, and $V_\kappa \models  \exists x \varphi (x,a)$, then by upwards absoluteness, $V\models \exists x \varphi (x,a)$.

For the general case $n>1$, assume $\kappa\in C^{(n)}$, and consider the $\Pi_n$-definable (with $a$ as a parameter) class of of structures of the form $\langle V_\alpha ,\in ,a\rangle$  with $\alpha \in C^{(n)}$.
\end{proof}

As Wilson \cite{W} shows that a cardinal  is remarkable if and only if it is weakly remarkable and belongs to $C^{(2)}$, if the least cardinal $\kappa$ that witnesses $\mathbf{\Pi_1}$-$\GSR$ is not remarkable, then is not weakly remarkable. Thus, the question is if it is provable in ZFC that the least cardinal $\kappa$ witnessing $\mathbf{\Pi_1}$-$\GSR$, if it exists, is weakly remarkable. Notice, however, that the proof of theorem \ref{firstremarkable} does show that if $\kappa$ is the least cardinal witnessing $\Pi_1$-$\GSR$, then either $\kappa$ is remarkable or there is a weakly remarkable cardinal below $\kappa$. Also, if $\kappa$ is the least cardinal witnessing $\mathbf{\Pi_1}$-$\GSR$, then either $\kappa$ is remarkable or there are unboundedly-many weakly remarkable cardinals below $\kappa$. 

\medskip

More generally, recall  (Definition \ref{n-rem} above) that a  cardinal $\kappa$ is \emph{$n$-remarkable}, for $n>0$, if for every $\lambda>\kappa$ in $C^{(n)}$, there is $\bar\lambda < \kappa$ also in $C^{(n)}$ such that in $V^{\Coll(\omega,\lt\kappa)}$, there is an elementary embedding $j:V_{\bar\lambda} \to V_\lambda$ with $j(\crit(j))=\kappa$. Equivalently, we may additionally require that for any given $a\in V_\lambda$, $a$ is in the range of $j$. 
A cardinal $\kappa$ is \emph{completely remarkable} if it is $n$-remarkable for every $n>0$.
Remarkable cardinals are precisely the $1$-remarkable cardinals.

As shown in \cite{BGS}, if $0^\sharp$ exists, then every Silver indiscernible is completely remarkable in $L$. Moreover, if $\kappa$ is $2$-iterable, then $V_\kappa$ is a model of ZFC in which there exist a proper class of completely remarkable cardinals. 

Theorem \ref{thmnrem} also holds for $n$-remarkable cardinals. Namely,

\begin{theorem}
\label{nrem}
The following are equiconsistent for $n>0$:
 \begin{enumerate}
 \item $\Pi_n$-$\GSR$
\item There exists a cardinal $\kappa$ that witnesses $\mathbf{\Pi_n}$-$\GSR$ 
\item There exists an  $n$-remarkable cardinal.
\end{enumerate}
\end{theorem}

As it turns out, $n+1$-remarkable cardinals correspond precisely to the virtual form of $C^{(n)}$-extendible cardinals. Namely,

\begin{definition}[\cite{BGS}]
A cardinal $\kappa$ is \emph{virtually extendible} if for every $\alpha>\kappa$ there is a generic elementary embedding $j:V_\alpha\to V_\beta$ such that $\crit(j)=\kappa$ and $j(\kappa)>\alpha$. 

A cardinal $\kappa$ is \emph{virtually $C^{(n)}$-extendible} if additionally $j(\kappa)\in C^{(n)}$.
\end{definition}

Note that virtually extendible cardinals are virtually $C^{(1)}$-extendible because $j(\kappa)$ must be inaccessible in $V$.

In contrast with the definition of extendible cardinal, in which the requirement that $j(\kappa)>\alpha$ is superfluous, in the definition of virtually extendible cardinal it is necessary. The reason is that while there is no non-trivial elementary embedding $j:V_{\lambda +2}\to V_{\lambda +2}$, such an embedding may exist generically (see \cite{BGS}).

\begin{theorem}[\cite{BGS}]
A cardinal $\kappa$ is virtually extendible if and only if it is $2$-remarkable. More generally, $\kappa$ is virtually $C^{(n)}$-extendible if and only if it is $n+1$-remarkable.
\end{theorem}

The requirement that $j(\kappa)>\alpha$ in the definition of virtually extendible cardinals suggests the following strengthening of $\GSR$. Let us say that an elementary embedding $j:V_\alpha \to V_\beta$ is \emph{overspilling} if $j$ has a critical point and $j(\crit(j))>\alpha$. For $\Gamma$ a lightface definability class, let:

\begin{quote}
\begin{itemize}
\item[$\Gamma$-$\SGSR$:] (\emph{$\Gamma$-Strong Generic  Structural Reflection}) There exists a  cardinal $\kappa$ that \emph{strongly generically-reflects} all $\Gamma$-definable classes $\mathcal{C}$ of natural structures, i.e.,  
 for every $A$ in $\mathcal{C}$ there exists $B$ in $\mathcal{C}\cap V_\alpha$ such that in $V^{\Coll(\omega,B)}$ there is an overspilling elementary embedding from $B$ into $A$.
\end{itemize}
\end{quote}
With the boldface version being:
\begin{quote}
\begin{itemize}
\item[] There exist  a proper class of cardinals $\kappa$ that \emph{strongly generically-reflect} all $\Gamma$-definable, with parameters in $V_\kappa$, classes $\mathcal{C}$ of natural structures.
\end{itemize}
\end{quote}

Then we have the following:

\begin{theorem}
The following are equivalent for every $n\geq 1$:
\begin{enumerate}
\item $\Pi_n$-$\SGSR$ 
\item There exists an $n+1$-remarkable cardinal
\item There exists a virtually $C^{(n)}$-extendible cardinal.
\end{enumerate}
\end{theorem}

\section{Beyond $\VP$}

We have seen that a variety of large cardinal notions, ranging from weakly inaccessible to Vop\v{e}nka's Principle, can be characterised as some form of Structural Reflection for classes of relational structures of some degree of complexity. The question is now if the same is true for large-cardinal notions stronger than  $\VP$, up to rank-into-rank embeddings, or even for large cardinals that contradict the Axiom of Choice (see \cite{BKW:LCBC}). This is largely a yet unexplored realm, although there are some very recent results showing that this is indeed the case. In the forthcoming \cite{BL}, we introduce a simple form of $\SR$, which we call \emph{Exact Structural Reflection $(\ESR)$}, and show that some natural large-cardinal notions  in the region between almost-huge and superhuge cardinals can be characterised in terms of $\ESR$. Also, sequential forms of $\ESR$ akin to generalised versions of Chang's Conjecture yield large-cardinal principles at the highest reaches of the known large-cardinal hierarchy, and beyond. We give next a brief summary of the results.

\medskip

Given infinite cardinals $\kappa<\lambda$ and a class $\Ce$ of structures of the same type,  let 
 \begin{quote}
\begin{itemize}
\item[$\ESR_\Ce(\kappa,\!\lambda)$:](\emph{Exact Structural Reflection}) For every $A\in \Ce$ of rank $\lambda$, there exists some $B\in \Ce$ of rank $\kappa$ and an elementary embedding form $B$ into $A$. 
\end{itemize}
\end{quote}
We let $\Gamma(P)$-$\ESR(\kappa,\lambda)$ denote the statement that $\ESR_\Ce (\kappa,\lambda)$ holds for every  class $\Ce$ of structures of the same type that is $\Gamma$-definable with parameters from $P$.

The general $\ESR$ principle restricted to classes of structures that are closed under isomorphic images is just equivalent to $\VP$:

\begin{theorem}[\cite{BL}]
 Over the theory {\rm{ZFC}}, the following schemata of sentences are equivalent: 
 \begin{enumerate}     
     \item For every class $\Ce$  of structures of the same type that is closed under isomorphic images,  there is a  cardinal $\kappa$ with  the property that $\ESR_\Ce(\kappa,\lambda)$ holds for all  $\lambda>\kappa$. 
     
         \item $\VP$.

 \end{enumerate}
\end{theorem}

However, even the principle $\Pi_1$-$\ESR (\kappa ,\lambda)$ holding for some $\kappa <\lambda$  already implies the existence of  large cardinals, the \emph{weakly exact cardinals}, whose consistency strength is beyond that of $\VP$.

\begin{definition}[\cite{BL}]
\label{defwexact}
 Given  a natural number $n>0$, an infinite cardinal $\kappa$ is \emph{weakly $n$-exact for a cardinal $\lambda>\kappa$} if for every $A\in V_{\lambda +1}$, there exists 
     a transitive, $\Pi_n(V_{\kappa+1})$-correct set $M$ with $V_\kappa\cup \{\kappa\} \subseteq M$,  a cardinal $\lambda'\in C^{(n-1)}$ greater than $\beth_\lambda$   
     and an elementary embedding $j:M\to H_{\lambda'}$ with $j(\kappa)=\lambda$ and $A\in\range(j)$.  
   
    If we further require that $j(\crit{j})=\kappa$, then we say that $\kappa$ is \emph{weakly parametrically $n$-exact for $\lambda$}.
 
\end{definition}

We  have the following equivalence:

\begin{theorem}[\cite{BL}]
 The following statements are equivalent for all cardinals $\kappa$ and all natural numbers $n>0$: 
 \begin{enumerate}
     
     \item $\kappa$ is the least cardinal such that $\Pi_n(V_\kappa)$-$\ESR(\kappa, \lambda)$ holds for some $\lambda$. 
     
     \item $\kappa$ is the least 
      cardinal that is weakly $n$-exact for some  $\lambda$. 
     
     \item $\kappa$ is the least cardinal that is weakly parametrically  $n$-exact for some  $\lambda$. 
 \end{enumerate}
\end{theorem}

In contrast with the $\SR$ principles considered in previous sections, the $\ESR$ principles for $\Pi_n$-definable  and $\Sigma_{n+1}$-definable  classes of structures are not equivalent.  Indeed, for $\Sigma_n$-definable classes, the relevant large cardinals are  the \emph{exact cardinals}:

\begin{definition}[\cite{BL}]
\label{defexact}
 Given a natural number $n$, an infinite  cardinal $\kappa$ is \emph{$n$-exact for some cardinal $\lambda>\kappa$} if for every $A\in V_{\lambda +1}$, there exists a cardinal $\kappa'\in C^{(n)}$ greater than $\beth_\kappa$, a cardinal  $\lambda'\in C^{(n+1)}$ greater than $\lambda$, an   $X\preceq H_{\kappa'}$ with  $V_\kappa \cup \{\kappa\}   \subseteq X$, and an elementary embedding $j:X\to H_{\lambda'}$ with  $j(\kappa)=\lambda$ and   $A\in\range{j}$. 

     If we further require that   $j(\crit{(j)})=\kappa$ holds,  then we say that $\kappa$ is \emph{parametrically $n$-exact for $\lambda$}.
 
\end{definition}

The characterization of $\ESR$ for $\Sigma_n$-definable classes of structures in terms of exact cardinals is now given by the following:

\begin{theorem}[\cite{BL}]
 The following statements are equivalent for all cardinals $\kappa$ and all natural numbers $n>0$: 
 \begin{enumerate}
     
     \item $\kappa$ is the least cardinal such that $\Sigma_{n+1}(V_\kappa)$-$\ESR(\kappa ,\lambda)$ holds for some $\lambda$. 
     
     \item $\kappa$ is the least cardinal that is  $n$-exact for some  $\lambda$. 
     
     \item $\kappa$ is the least cardinal that is  parametrically  $n$-exact for some  $\lambda$. 
 \end{enumerate}
\end{theorem}

The strength of weakly $n$-exact and $n$-exact cardinals, and therefore also of their corresponding equivalent forms of $\ESR$, goes  beyond $\VP$, for as shown in  \cite{BL} they imply the existence of almost huge cardinals:\footnote{Recall that a cardinal $\kappa$ is \emph{almost huge} if there exists a transitive class $M$ and a non-trivial elementary embedding $j:V\to M$ with  $\crit{j}=\kappa$ and ${}^{{<}j(\kappa)}M\subseteq M$.
We then say that a 
cardinal $\kappa$ is almost huge  with \emph{target $\lambda$} if there exists such a $j$  with $j(\kappa)=\lambda$. } 
If $\kappa<\lambda$ are cardinals such that $\kappa$ is either parametrically $0$-exact for $\lambda$ or weakly parametrically $1$-exact for $\lambda$, then the set of cardinals  that are almost huge with target $\kappa$ is stationary in $\kappa$.
 Also, if $\kappa$  is parametrically $0$-exact for some cardinal $\lambda>\kappa$, then it is almost huge with target $\lambda$. 
 
As for upper bounds,  if $\kappa$ is huge with target $\lambda$, then it is weakly parametrically $1$-exact for $\lambda$.   Hence, $\Pi_1(V_\mu )$-$\ESR(\mu,\nu)$ holds for some $\mu \leq \kappa$ and $\nu >\mu$.   Moreover,  $\Pi_1(V_\kappa)$-$\ESR(\kappa , \lambda')$ holds in $V_\lambda$, for some $\lambda'$. However, if $\kappa$ is the least huge cardinal, then $\kappa$ is not $1$-exact for any cardinal $\lambda>\kappa$. The best upper bound for the consistency strength of exact cardinals is given by the following:
 
  \begin{prop}[\cite{BL}]
  If $\kappa$ is a $2$-huge cardinal,\footnote{I.e., there is an elementary embedding $j:V\to M$ with $M$ transitive, $\crit{(j)}=\kappa$, and $^{j^2(\kappa)}M\subseteq M$.} then there exists an inaccessible cardinal $\lambda>\kappa$ and a cardinal $\rho>\lambda$ such that $V_\rho$ is a model of \rm{ZFC} and, in $V_\rho$, the cardinal $\kappa$ is weakly parametrically $n$-exact for $\lambda$, for all $n>0$. 
 \end{prop}
 
  As for direct implication, the best known upper bound for the existence of exact cardinals is given by the following:
 
\begin{prop}
 [\cite{BL}]
  Let $\kappa$ be an $I3$-cardinal\footnote{I.e.,  the critical point of a non-trivial elementary embedding $j:V_{\delta}\to V_{\delta}$, for some limit ordinal $\delta$. Then $V_\delta$ is a model of \rm{ZFC} and the sequence $\seq{j^m(\kappa)}{m<\omega}$  is cofinal in $\delta$.} $j:V_\delta\to V_\delta$. If  $l,m,n<\omega$, then, in $V_\delta$, the  cardinal $j^l(\kappa)$ is parametrically $n$-exact for $j^{l+m+1}(\kappa)$.    
 \end{prop}

The existence of an $I3$-cardinal is a very strong  principle which implies the consistency of  of $n$-huge cardinals, for every $n$, and much more (see \cite[24]{Kan:THI}, also \cite[Theorem 7.1]{Ba:CC}). Yet even stronger  large-cardinal principles bordering the inconsistency with ZFC are implied by the following sequential forms of $\ESR$, also introduced in \cite{BL}.

\subsection{Sequential $\ESR$}
 Let $0<\eta\leq\omega$  and  let $\calL$ be a first-order language containing unary predicate symbols $\vec{P}=\seq{\dot{P}_i}{i<\eta}$.

  Given a sequence $\vec{\mu}=\seq{\mu_i}{i<\eta}$ of cardinals with supremum $\mu$, an $\calL$-structure $A$ has \emph{type $\vec{\mu}$ (with respect to $\vec{P}$)} if the universe  of $A$ has rank $\mu$ and $\rank{(\dot{P}_i^{A})}=\mu_i$ for all $i<\eta$.

 Given a class $\Ce$ of $\calL$-structures and  a strictly increasing sequence $\vec{\lambda}=\seq{\lambda_i}{i<1+\eta}$ of cardinals,  let
 \begin{quotation}
\begin{itemize}
\item[$\ESR_{\Ce}$$ ({\vec{\lambda}})$:](\emph{Sequential $\ESR$)} For every structure $B$ in $\Ce$ of type  $\seq{\lambda_{i+1}}{i<1+\eta}$, there exists an elementary embedding of a structure $A$ in $\Ce$ of type $\seq{\lambda_i}{i<\eta}$ into $B$. 
\end{itemize}
\end{quotation}

The large-cardinal notions corresponding to sequential $\ESR$ are the  sequential analogs of weakly exact and exact cardinals   (given in definitions \ref{defwexact} and \ref{defexact}).   Namely,

\begin{definition}[\cite{BL}]
 Let $0<\eta\leq\omega$ and let $\vec{\lambda}=\seq{\lambda_m}{m<\eta}$ be a strictly increasing sequence of  cardinals with supremum $\lambda$. 
 
 \begin{enumerate}
  \item Given $0<n<\omega$, a cardinal $\kappa <\lambda_0$ is \emph{weakly $n$-exact for $\vec{\lambda}$} if for every $A\in V_{\lambda+1}$, there is a cardinal $\rho$, a  transitive,   $\Pi_n(V_{\rho+1})$-correct set $M$ with $V_\rho \cup \{\rho\}\subseteq M$, a cardinal $\lambda'\in C^{(n-1)}$ greater than $\beth_\lambda$    and an e. e. $j:M\to H_{\lambda'}$ with  $A\in\range{j}$, $j(\rho)=\lambda$,  $j(\kappa)=\lambda_0$ and  $j(\lambda_{m-1})=\lambda_m$, all $m$.

   If we further require that $j(\crit{j})=\kappa$, then we say that $\kappa$ is \emph{parametrically weakly $n$-exact for $\vec{\lambda}$}.

     \item Given $n<\omega$, a cardinal $\kappa<\lambda_0$ is \emph{$n$-exact for $\vec{\lambda}$}  if for every $A\in V_{\lambda +1}$, there is a cardinal $\rho$, a cardinal $\kappa'\in C^{(n)}$ greater than $\beth_\rho$, a cardinal $\lambda' \in C^{(n+1)}$ greater than $\lambda$,  an  $X\preceq H_{\kappa'}$ with $V_{\rho}\cup \{\rho\}\subseteq X$, and 
 an e. e. $j:X\to H_{\lambda'}$ with $A\in\range{j}$, $j(\rho)=\lambda$, $j(\kappa)=\lambda_0$ and $j(\lambda_{m-1})=\lambda_m$,  all $m$.

   If we further require that $j(\crit{j})=\kappa$, then we say that \emph{$\kappa$ is parametrically $n$-exact for $\vec{\lambda}$}.

 \end{enumerate}
\end{definition}

Then we have the following equivalences:

\begin{theorem}[\cite{BL}]
 Let $0<n<\omega$, let $0<\eta\leq\omega$ and let $\vec{\lambda}=\seq{\lambda_i}{i<1+\eta}$ be a strictly increasing sequence of cardinals. 
  \begin{enumerate}
   \item The cardinal $\lambda_0$ is weakly $n$-exact for $\seq{\lambda_{i+1}}{i<\eta}$ if and only if  $\Pi_n$-$\ESR$$ ({\vec{\lambda}})$ holds. 
   
   \item If $\lambda_0$ is weakly parametrically $n$-exact for $\seq{\lambda_{i+1}}{i<\eta}$, then $\Pi_n(V_{\lambda_0})$-$\ESR$$(\vec{\lambda})$ holds. \qed 
  \end{enumerate}
  
  Also,
 
  \begin{enumerate}
   \item The cardinal $\lambda_0$ is  $n$-exact for $\seq{\lambda_{i+1}}{i<\eta}$ if and only if  $\Sigma_{n+1}$-$\ESR$$(\vec{\lambda})$ holds. 
   
   \item If $\lambda_0$ is  parametrically $n$-exact for $\seq{\lambda_{i+1}}{i<\eta}$, then $\Sigma_{n+1}(V_{\lambda_0})$-$\ESR$$(\vec{\lambda})$ holds. \qed 
  \end{enumerate}
\end{theorem}

%
%
%
%
%
%
%
%
%
%

\medskip

In the case of finite sequences $\vec{\lambda}$ of length $n$, the sequentially 1-exact cardinals correspond roughly to $n$-huge cardinals. More precisely (\cite{BL}): 
 If $\kappa$ is an $n$-huge cardinal, witnessed by an elementary embedding $j:V\to M$, then  $\kappa$ is weakly parametrically $1$-exact for the sequence $\seq{j^{m+1}(\kappa)}{m<n}$. Also, if $\kappa$ is a cardinal and $\vec{\lambda}=\seq{\lambda_m}{m\leq n}$ is a sequence of cardinals such that $\kappa$ is either weakly $1$-exact for $\vec{\lambda}$ or $0$-exact for $\vec{\lambda}$, then  some cardinal less than $\kappa$ is $n$-huge.

As for infinite sequences $\vec{\lambda}=\seq{\lambda_m}{m<\omega}$, there is a dramatic increase in consistency strength,  as shown by the following facts proved in \cite{BL}: 
Let $\lambda$ be the supremum of $\vec{\lambda}$ and let $\kappa<\lambda_0$ be a cardinal.
 If $\kappa$ is either   weakly $1$-exact  for $\vec{\lambda}$ or  $0$-exact for $\vec{\lambda}$, then there exists an  $I3$-embedding $j:V_\lambda\to V_\lambda$. Also, if $\kappa$ is either  parametrically weakly $1$-exact  for $\vec{\lambda}$ or parametrically $0$-exact for $\vec{\lambda}$, then the set  $I3$-cardinals is stationary in $\kappa$. 
 
To prove the existence of a weakly parametrically $1$-exact cardinal, for some infinite sequence $\vec{\lambda}$, the best known upper bound is an $I1$-cardinal\footnote{I.e., the critical point of a non-trivial elementary embedding $j:V_{\lambda +1}\to V_{\lambda +1}$ for some limit ordinal $\lambda$.}  (\cite{BL}): If $\kappa$ is  an $I1$-cardinal and $k>0$ is a natural number, then $\kappa$ is weakly parametrically $1$-exact for the sequence $\seq{j^{k(m+1)}(\kappa)}{m<\omega}$.  
 In particular, for $k=1$, $\kappa$ is weakly parametrically $1$-exact for $\seq{j^{(m+1)}(\kappa)}{m<\omega}$, hence $\Pi_1(V_\kappa)$-$\ESR$$(\vec{\lambda})$ holds.

%
%
%
%
%
%
%
\bigskip

Many open questions remain (see \cite{BL}), the most pressing one being  the consistency with ZFC of the principle $\Sigma_2$-$\ESR$$(\vec{\lambda})$  for some sequence $\vec{\lambda}$ of length $\omega$.

\newpage


\section{Summary}
The following tables summarize the results exposed in previous sections. Presented in this form, the equivalences between the various kinds of $\SR$ and (mostly already well-known) different large cardinal notions, illustrate the fact that $\SR$ is a general reflection principle that underlies (many stretches of) the large cardinal hierarchy, thus unveiling its concealed uniformity.
Table 1 encompasses the region of the large-cardinal hierarchy comprised between supercompact and VP, Table 2 the region between strong and ``$\OR$ is Woodin", and Table 3 the region between globally superstrong and $C^{(n)}$-globally superstrong, all $n$. The tables are read as, e.g., $\Gamma$-$\SR$ for $\Gamma$ being $\Pi_1$ or $\Sigma_2$ is equivalent to the existence of a supercompact cardinal (Table 1). For the boldface definability classes we have the equivalence of $\Gamma$-$\SR$ with a proper class of the corresponding large cardinals. In the limit cases, i.e., VP in Table 1, ``$\OR$ is Woodin" in Table 2, and $C^{(n)}$-globally superstrong , all $n$, in Table 3, the lightface and the boldface versions are equivalent.

\bigskip

\begin{center}
Table 1
\vspace{.3cm}

\begin{tabular}{ |c|c| } 
\hline
$\Gamma$ & $\SR$  \\
\hline
\hline

$\Sigma_1$ & ZFC  \\ 
\hline
$\Pi_1$, $\Sigma_2$ & Supercompact  \\
\hline
$\Pi_2$, $\Sigma_3$ & Extendible   \\ 
\hline
$\Pi_3$, $\Sigma_4$ & $C^{(2)}$-Extendible    \\ 
\hline
$\vdots$ & $\vdots$  \\
\hline
$\Pi_n$, $\Sigma_{n+1}$ & $C^{(n-1)}$-Extendible   \\
\hline
$\vdots$ & $\vdots$  \\
\hline
$\Pi_n$, all $n$ & $\VP$   \\
\hline
\end{tabular}
\end{center}

\bigskip

\begin{center}
Table 2
\vspace{.3cm}

\begin{tabular}{ |c|c| } 
\hline
Complexity &   $\PSR$ \\
\hline
\hline

$\Sigma_1$  & ZFC \\ 
\hline
$\Pi_1$, $\Sigma_2$   & Strong \\
\hline
$\Pi_2$, $\Sigma_3$   & $\Pi_2$-Strong \\ 
\hline
$\Pi_3$, $\Sigma_4$    & $\Pi_3$-Strong \\ 
\hline
$\vdots$ & $\vdots$   \\
\hline
$\Pi_n$, $\Sigma_{n+1}$ &   $\Pi_n$-Strong \\
\hline
$\vdots$ & $\vdots$   \\
\hline
$\Pi_n$, all $n$ &   $\OR$ is Woodin \\
\hline
\end{tabular}
\end{center}

\medskip

We also have the equivalence between $\Pi_1$-$\PSR_1$ and the existence of a measurable cardinal (Theorem \ref{PSRmeasurable}).

\bigskip

\begin{center}
Table 3 
\vspace{.3cm}

\begin{tabular}{ |c|c| } 
\hline
Complexity &   $\SPSR$ \\
\hline
\hline

$\Sigma_1$  & ZFC \\ 
\hline
$\Pi_1$, $\Sigma_2$   & Globally Superstrong \\
\hline
$\Pi_2$, $\Sigma_3$   & $C^{(2)}$-Globally Superstrong \\ 
\hline
$\Pi_3$, $\Sigma_4$    & $C^{(3)}$-Globally Superstrong \\ 
\hline
$\vdots$ & $\vdots$   \\
\hline
$\Pi_n$, $\Sigma_{n+1}$ &   $C^{(n)}$-Globally Superstrong \\
\hline
$\vdots$ & $\vdots$   \\
\hline
$\Pi_n$, all $n$ &   $C^{(n)}$-Globally Superstrong, all $n$ \\
\hline
\end{tabular}
\end{center}

\bigskip

Table 4 summarizes the results on $\SR$ relative to inner models. The classes $\Ce$, $\Ce_X$, $\Ce^U$ and $\Ce^U_X$ are defined in section \ref{sectionL}. Similar results should hold for canonical inner models for stronger large cardinals, and their corresponding \emph{sharps}.

\medskip
\begin{center}
Table 4
\vspace{.3cm}

\begin{tabular}{ |c|c|c|c| } 
\hline
Class & Inner Model $M$ & $\SR(M)$ \\
\hline
\hline

$\Sigma_1$ & Any & ZFC \\ 
\hline
$\Ce$  & $L$ & $0^\sharp$ exists \\
\hline
$\Ce_X$  & $L[X]$ & $X^\sharp$ exists \\ 
\hline
$\Ce^U$  & $L[U]$  & $0^\dagger$ exists \\ 
\hline
$\Ce^U_X$  & $L[U,X]$  & $X^\dagger$ exists \\
\hline
$\vdots$ & $\vdots$ & $\vdots$ \\
\hline

\end{tabular}
\end{center}

\bigskip

Tables 5 and 6 below summarize the results characterizing large cardinals of consistency strength below $0^\sharp$ in terms of restricted $\SR$ and generic $\SR$, respectively. 

\medskip

\begin{center}
Table 5

\vspace{.3cm}
\begin{tabular}{ |c|c|c|c| } 
\hline
Complexity &   $\SR^-$ & $\SR^- +{\rm GCH}$ \\
\hline
\hline

$\Sigma_1(\PW)$, $\Sigma_2$  & Weakly shrewd & \\ 
\hline
$\Sigma_1(\Cd)^\ast$   & Weakly inaccessible & Inaccessible \\
\hline
$\Sigma_1(\Cd, \alpha$-$\WI)^\ast$   & $\alpha$-Weakly inaccessible & $\alpha$-Inaccessible \\ 
\hline
$\Sigma_1(\Rg)^\ast$   & Weakly Mahlo & Mahlo \\ 
\hline
$\Sigma_1(\Cd, \WC)^\ast$ &  & Weakly-compact  \\
\hline

\end{tabular}
\end{center}

\bigskip

\begin{center}
Table 6

\vspace{.3cm}
\begin{tabular}{ |c|c|c|c| }  
\hline
Complex. &   $\GSR$ & $\GSR +$No $\omega$-Erd\"os  in $L$ & $\SGSR$ \\
\hline
\hline

$\Pi_1$, $\Sigma_2$  & Almost- & Remarkable & \\ 
   & Remarkable  & Almost-remarkable & \\ 
      &   & Weakly-remarkable & \\ 
\hline
$\Pi_n$, $\Sigma_{n+1}$   &  Almost- & & $n+1$-Remarkable   \\
   & $n$-remarkable  & & Virt. $C^{(n)}$-Extend.  \\ 

\hline

\end{tabular}
\end{center}
\bigskip

Finally, Tables 7-8 cover some equivalences in the region above Vop\v{e}nka's Principle.

\begin{center}
Table 7

\vspace{.3cm}
\begin{tabular}{ |c|c| }  
\hline
Complex. &   $\ESR$  \\
\hline
\hline

$\Pi_1$  & Weakly $1$-exact   \\ 
   & (Between almost huge and huge)    \\ 
   
\hline

$\Pi_n$  & Weakly $n$-exact   \\ 
  &(Consistency-wise below $2$-huge)       \\ 
   
\hline
$ \Sigma_{n+1}$   &  $n$-Exact     \\
   & ($I3$-embedding  is an upper bound)   \\ 
\hline

\end{tabular}
\end{center}

\bigskip

\begin{center}

Table 8

\vspace{.3cm}
\begin{tabular}{ |c|c|c| }  
\hline
Complex. &    $\ESR$$ (\vec{\lambda})$,  ${\rm{lh}}(\vec{\lambda})=\eta +1$ &  $\ESR$$ (\vec{\lambda})$,  ${\rm{lh}}(\vec{\lambda})=\omega$ \\
\hline
\hline

$\Pi_1$  &  Weak. $1$-exact for $\vec{\lambda}$ of lh. $\eta$ & Weak. $1$-exact for $\vec{\lambda}$ of lh. $\omega$ \\ 
      & (Implied by $\eta$-huge)&  (Implies $I3$-cardinals. \\ 
   & & $I1$-cardinal  an upper bound) \\
\hline

$\Pi_n$  &  Weak. $n$-exact for $\vec{\lambda}$ of lh. $\eta$ & Weak. $n$-exact for $\vec{\lambda}$ of lh. $\omega$\\ 
    & &   \\ 
   
\hline
$ \Sigma_{n+1}$   &   $n$-Exact for $\vec{\lambda}$ of length $\eta$  & $n$-Exact for $\vec{\lambda}$ of length $\omega$ \\
    & &   \\ 
    & &   \\
\hline

\end{tabular}
\end{center}

\bibliographystyle{alpha} 
\bibliography{../masterbiblio}

\end{document}